\documentclass{amsart}

\usepackage[backref=page]{hyperref}
\hypersetup{
    colorlinks=true,
    linkcolor=blue,
    filecolor=blue,      
    urlcolor=blue,
    citecolor=blue,
}
\usepackage[shortalphabetic]{amsrefs}
\usepackage{amsmath}
\usepackage{amsthm}
\usepackage{amssymb}
\usepackage{graphicx}
\usepackage{epstopdf}
\usepackage{float}
\usepackage{comment}
\usepackage{booktabs}
\usepackage[labelformat=simple]{subcaption}
\usepackage{url}
\usepackage{tikz}
\usepackage{pgf}
\usepackage{wrapfig}

\numberwithin{equation}{section}
\newtheorem{thm}[equation]{Theorem}
\newtheorem{thmalpha}{Theorem}

\newtheorem{lem}[equation]{Lemma}
\newtheorem{prop}[equation]{Proposition}
\newtheorem{cor}[equation]{Corollary}
\newtheorem{definition}[equation]{Definition}

\newtheorem*{definition*}{Definition}

\newtheorem*{prob*}{Problem}

\theoremstyle{definition}
\newtheorem{ex}[equation]{Example}
\newtheorem{remark}[equation]{Remark}

\DeclareMathOperator{\Aut}{Aut}

\DeclareMathOperator{\im}{{im}}
\DeclareMathOperator{\Lat}{{Lat}}

\DeclareMathOperator{\Nor}{Nor}

\DeclareMathOperator{\support}{supp}
\DeclareMathOperator{\Zmon}{\sf Z}
\DeclareMathOperator{\Nmon}{\sf N}

\newcommand{\N}{\mathbb{N}}

\newcommand{\respart}{R_E}
\newcommand{\supp}[2]{\support_{#1}(#2)}
\newcommand{\psupp}[2]{\support_{#1}^+(#2)}

\author{Joshua Maglione}\address{
	Fakult\"at f\"ur Mathematik\\
   Universit\"at Bielefeld\\
   Postfach 100131\\
	D-33501 Bielefeld, Germany
}
\email{jmaglione@math.uni-bielefeld.de}
\title{Filters compatible with isomorphism testing}
\date{}
\thanks{This research was supported in part by NSF grant DMS 1620454.}

\allowdisplaybreaks 

\begin{document}

\begin{abstract}
   Like the lower central series of a nilpotent group, filters generalize the
   connection between nilpotent groups and graded Lie rings. However, unlike the
   case with the lower central series, the associated graded Lie ring may share
   few features with the original group: e.g.~the associated Lie ring can be
   trivial or arbitrarily large. We determine properties of filters such that
   every isomorphism between groups is induced by an isomorphism between graded
   Lie rings. 
\end{abstract}

\maketitle

\section{Introduction}

There are several progressively more general methods to associate a Lie ring to
a nilpotent group, see \cites{Higman:Lie-techniques, Khukhro:p-Auts}. While the
applications vary, a common theme is to make group-theoretic problems easier by
employing linear algebra in the context of the Lie ring. In particular, this
helps in the study of isomorphism and automorphism problems for groups
\cites{BOW:graded-algebras, ELGO:auts, Higman:auto-fixed, M:efficient-filters,
M:small-auts}. A recent development, described in \cite{W:char}, generalizes
approaches from Magnus~\cites{Magnus:BCH-formula, Magnus:Lie-ring} and
Lazard~\cite{Lazard:N-Series}. J.B.\ Wilson defined filters as a means to allow
refinements of the well-studied upper and lower central series associated to
nilpotent groups, while still connected to a graded Lie ring. We describe a
process of lifting isomorphisms between the associated Lie rings to (potential)
isomorphisms between the groups, which has applications to isomorphism testing
of finite $p$-groups.

\emph{Throughout, we assume that $G$ is a finitely generated nilpotent group and
$M$ is a finitely generated commutative monoid}. An indexed family $\{\phi_s\leq
G\mid s\in M\}$ is a \emph{filter} if for all $s,t\in M$, $[\phi_s, \phi_t] \leq
\phi_{s+t} \leq \phi_s\cap\phi_t$, where $[,]$ is the commutator in $G$. The
filter is \emph{finite} if the set of subgroups is finite. The use of more
general monoids rather than just $\N$ allows one to index more of the subgroup
lattice while maintaining an associated $M$-graded Lie ring
\begin{align}\label{eqn:Lie-ring}
   L(\phi) = \bigoplus_{s\ne 0} \phi_s/ \langle \phi_{s+t}\mid t\ne 0\rangle. 
\end{align}

Form a  category of $M$-filtered groups, $M$-\textsf{Group}, together with all
group homomorphisms $\alpha:G\rightarrow H$ such that for all $s\in M$,
$(\phi_s(G))^\alpha \leq \phi_s(H^\alpha)$. A functor $\phi: \textsf{Group}
\rightarrow M\text{-}\textsf{Group}$ is an \emph{$M$-filter functor}. For
example, the lower central series is an $\N$-filter functor as every subgroup in
every $\N$-filter of a group $G$ is a verbal subgroup. An important class of
filter functors are those where isomorphisms $\alpha: G\rightarrow H$ are lifted
from isomorphisms $\hat{\alpha} : L(\phi(G))\rightarrow L(\phi(H))$. This is
especially desireable in the context of computation and, specifically, the Group
Isomorphism Problem~\cites{BGLQW:WL-iso, BOW:graded-algebras,
M:efficient-filters}. In the next subsection, we define a large class of filters
with this property, but we first state the main result.

An \emph{inertia-free, faithful filter} $\phi$ of a group $G$ is one that
satisfies some non-degeneracy conditions and gives rise to a generating set
$\mathcal{X}\subseteq G$ (detailed Definitions~\ref{def:inert}
and~\ref{def:faithful}), which is analogous to the connection between polycyclic
series and polycyclic generating set; see~\cite{Sims:book}*{Chapter~9}. 

\begin{thmalpha}\label{thm:Main0}
   If $\phi$ is a finite, inertia-free, faithful-filter functor, then all
   isomorphisms between groups $G$ and $H$ are lifts of $M$-graded isomorphisms
   between $L(\phi(G))$ and $L(\phi(H))$. 
\end{thmalpha}

\subsection{Applications to group isomorphism}

One of the driving motivations for Theorem~\ref{thm:Main0} comes from the
isomorphism problem of groups, see for example \cites{BCQ:Poly, GQ:group-iso}.
General-purpose algorithms to decide isomorphism or compute automorphism groups
of finite nilpotent groups rely on induction and the ability to construct
characteristic subgroups \cites{BGLQW:WL-iso, CH:iso, ELGO:auts,
O'Brien:isomorphism}. The algorithms for finite nilpotent groups reduce to
finite $p$-groups. The difficult case is when $G$ has precisely three known
characteristic subgroups, namely $G$, $G'$, and $\langle 1\rangle$.

Because of Theorem~\ref{thm:Main0}, we can leverage the structure of the Lie
algebra $L(\phi)$ to constrain the possible automorphisms of $G$. For example, a
characteristic subalgebra of $L(\phi)$ induces a characteristic subgroup of $G$,
which can be used to refine $\phi$ into a new filter capturing richer structure.
Additionally, if it is more efficient to calculate graded Lie automorphisms of
$L(\phi)$ than to construct $\Aut(G)$ recursively by going down the lower
$p$-central series, this may be a starting point to construct automorphisms of
$G$. For example, there is a large family of graded algebras $L$ such that the
graded Lie automorphisms can be constructed in time polynomial in $|L|$, see
\cite{BOW:graded-algebras} for details. 

Filters---in particular, their associated $M$-graded Lie rings---are a
significant resource for constructing characteristic subgroups efficiently, see
the examples in Section~\ref{sec:examples-section}
and~\cites{M:efficient-filters, M:classical, W:char}. The main benefit is that
the inclusion of one new subgroup in a filter can drastically change the filter
and the associated graded Lie algebra. More work is needed to understand how to
best refine filters, but this opens the door to efficient recursive methods,
uncovering characteristic subgroups. The characteristic features of the
associated $M$-graded Lie rings offer the possibility of exponential speedups in
the realm of isomorphism testing for finite $p$-groups. 

We also view Theorem~\ref{thm:Main0} as a necessary first-step to developing a
general, efficient filter refinement algorithm, similar to
how~\cite{M:efficient-filters}*{Theorem~6} was used to construct an efficient
filter refinement algorithm over totally ordered monoids. Such a refinement
algorithm would be applied recursively until we have exhausted all possible
refinements or the automorphism group of the Lie ring is small enough to run
through. And if at each step, the algorithm can guarantee the filter is
inertia-free and faithful, then the final task is to try to lift all Lie
isomorphisms.

\subsection{Detailed description of results}

We now detail the classes of filters in the statement of
Theorem~\ref{thm:Main0}.

\begin{definition}\label{def:filter}
   For a pre-ordered monoid $M=\langle M,\preceq\rangle$, with minimal element
   $0$, and a group $G$, an $(M,G)$-\emph{filter} is a function from $M$ into
   the set of normal subgroups of $G$ such that $\langle\phi_s\mid s\ne 0\rangle
   = G$, $\bigcap_{s\in M}\phi_s = 1$, and for all $s,t\in M$,
   \begin{align*}
      [\phi_s,\phi_t]&\leq \phi_{s+t} & s\preceq t &\implies \phi_s\geq\phi_t.
   \end{align*}
\end{definition}

We call a pre-ordered (commutative) monoid, with minimal element $0$,
\emph{conically pre-ordered}. An $(M, G)$-filter $\phi$ is \emph{finite} if
$\im(\phi)$ is a finite set. We call $\phi$ \emph{degenerate} if either $\langle
\phi_s\mid s\ne 0\rangle \ne G$ or $\bigcap_{s\in M}\phi_s \ne 1$. The
\emph{boundary filter} of an $(M,G)$-filter $\phi$, denoted $\partial\phi$, is
the $(M,G)$-filter such that for all $s\in M$, $\partial\phi_s = \langle
\phi_{s+t} \mid t\ne 0\rangle$. Then the $M$-graded Lie ring
from~\eqref{eqn:Lie-ring} is $L(\phi) = \bigoplus_{s\ne 0}
\phi_s/\partial\phi_s$. We call $\mathcal{Y}\subseteq L(\phi)$ a \emph{graded
generating set} if $\mathcal{Y}$ generates the abelian group $L(\phi)$ and for
each $y\in \mathcal{Y}$ there exists $s\in M$ such that $y\in L_s(\phi):=
\phi_s/\partial\phi_s$. 

\begin{thmalpha}\label{thm:Main}
   If $\phi$ is a finite $(M,G)$-filter, with partial order
   $\preceq$, then there exists a conically partially-ordered monoid $M'$, an
   $(M',G)$-filter $\theta$, a graded generating set $\mathcal{Y}$ for the
   abelian group $L(\theta)$, and a surjection of sets $\pi_{\mathcal{Y}} :
   L(\theta) \rightarrow G$ such that $\im(\phi)\subseteq\im(\theta)\subset 2^G$
   and $\pi_{\mathcal{Y}}(\mathcal{Y})$ generates $G$. 
\end{thmalpha}

The requirement that $\preceq$ is a partial order can be replaced with
additional assumptions on $\phi$ (Theorems~\ref{thm:general-progressive}
\&~\ref{thm:surjection}). The filter $\theta$ constructed in
Theorem~\ref{thm:Main} is \emph{inertia-free}, defined below, and this property
guarantees $\pi_{\mathcal{Y}}$ is a surjection. For an $(M,G)$-filter $\phi$,
define an ascending chain of subsets $\mathfrak{B}$ in $\im(\phi)$ such that
$\mathfrak{B}_0=\{\langle 1\rangle\}$ and for $i\geq 0$,
\begin{align}\label{eqn:inertia-set}
   \mathfrak{B}_{i+1} = \{ \phi_s \mid \exists B\subseteq \mathfrak{B}_i,\, \partial\phi_s=\langle B\rangle \}.
\end{align}
\begin{definition}\label{def:inert}
   A filter $\phi$ is \emph{inertia-free} if $\bigcup_{i\geq
   0}\mathfrak{B}_i=:\mathfrak{B}=\im(\phi)$. A subgroup $H\in\im(\phi)$ is an
   \emph{inert} subgroup if $H\in\im(\phi)\setminus\mathfrak{B}$. 
\end{definition}

For example, if $\gamma$ is the $(\N, G)$-filter given by the lower central
series (with $\gamma_0=\gamma_1=G$), then $\partial\gamma_k=\gamma_{k+1}$. Since
$\gamma_k\ne\gamma_{k+1}$ unless either $k=0$ or $\gamma_k=\langle 1\rangle$, it
follows that $\mathfrak{B}_i = \{\gamma_j~|~ j\geq c-i\}$, where $c$ is the
nilpotency class of $G$. The integer $i$ in~\eqref{eqn:inertia-set} measures the
``distance'' between $H\in\mathfrak{B}_i\setminus\mathfrak{B}_{i-1}$ and
$\langle1\rangle$ by taking boundaries. Inert subgroups are those that never
reach $\langle 1\rangle$.

Further properties of an inertia-free filter are required to ensure that
$\pi_{\mathcal{Y}}$ from Theorem~\ref{thm:Main} is injective. The closure of
$\im(\phi)\subset 2^G$ under intersections and products will be denoted by
$\Lat(\phi)$.

\begin{definition}\label{def:faithful}
   An $(M, G)$-filter $\phi$ is \emph{faithful} if there exists
   $\mathcal{X}\subseteq G$ such that 
   \begin{enumerate}
      \item[$(i)$] for all $H\in\Lat(\phi)$, $\langle H\cap \mathcal{X}\rangle =
      H$, 
      \item[$(ii)$] the map $H\mapsto H\cap \mathcal{X}$ is a lattice embedding
      from $\Lat(\phi)$ into the subset lattice of $\mathcal{X}$, and 
      \item[$(iii)$] for each $x\in \mathcal{X}$, there exists a unique $s\in M$
      such that $x\in\phi_s\setminus\partial\phi_s$.
   \end{enumerate}
\end{definition}

As $G$ is a finitely generated nilpotent group, it is a polycyclic group. Every
polycyclic group $G$ has a \emph{polycyclic generating sequence (pcgs)}
$(a_1,\dots, a_n)$ such that for $G_i=\langle a_i,\dots, a_n\rangle$, the
factors $G_i/G_{i+1}$ are cyclic for all $i\in\{1,\dots, n-1\}$, see for example
\cite{Sims:book}*{Chapter~9}. Along with the map given in
Theorem~\ref{thm:Main}, we use properties of pcgs to prove the next theorem.

\begin{thmalpha}\label{thm:Main2}
   If $\phi$ is a finite, faithful, and inertia-free $(M, G)$-filter, then there exists a bijection between $L(\phi)$ and $G$ that maps a pcgs of $L(\phi)$, as an abelian group, to a pcgs of $G$.
\end{thmalpha}

In particular, if the filter $\theta$ from Theorem~\ref{thm:Main} is faithful,
then the map $\pi_{\mathcal{Y}}$ is a bijection. When the monoid $M$ is cyclic,
like in the $(\N, G)$-filter $\gamma$ given by the lower central series with
$\gamma_0 = \gamma_1 = G$, every filter is faithful and inertia-free. In the
special case that the monoid $M$ is totally ordered, one can construct faithful,
inertia-free filters in a simpler way, see~\cite{M:efficient-filters}.
Therefore, we are interested in pre-orders that are not total orders, but we do
not require this.

We are interested in lifting isomorphisms between graded Lie rings to
isomorphisms between groups and Theorem~\ref{thm:Main2} is critical to that
objective. For an $(M, G)$-filter $\phi$, define its \emph{border set} as 
\begin{align}\label{eqn:relevant-set}
   \mathcal{I}_\phi = \left\{ s\in M ~\middle|~ \partial\phi_s \ne \phi_s \text{ or } \phi_s = 1 \right\}.
\end{align}
Note that $L(\phi)\cong\bigoplus_{s\in\mathcal{I}_\phi}\phi_s/\partial\phi_s$ as
abelian groups.

If $G$ and $H$ are groups and $\alpha:G\rightarrow H$ is a homomorphism, then
$\alpha$ induces a homomorphism between $M$-\textsf{Groups}. That is, if $\phi$
is a possibly degenerate $(M, G)$-filter, then $\phi^{\alpha}$ is a possibly
degenerate $(M, H)$-filter, where for all $s\in M$, $(\phi^{\alpha})_s =
(\phi_s)^\alpha$. Furthermore, $\alpha$ induces an $M$-graded Lie homomorphism
$\hat{\alpha}: L(\phi)\rightarrow L(\phi^\alpha)$. If $\alpha$ is an isomorphism
and $\phi$ is inertia-free and faithful, then $\phi^\alpha$ is inertia-free and
faithful and $\hat{\alpha}$ is an isomorphism.

The goal now is to produce group isomorphisms given only $M$-graded Lie
isomorphisms when possible. That is, if $\phi$ and $\theta$ are $(M, G)$- and
$(M, H)$-filters and $\beta: L(\phi)\rightarrow L(\theta)$ is an $M$-graded
isomorphism, determine if there exists an isomorphism $\alpha : G\rightarrow H$
such that $\phi^\alpha = \theta$ and $\hat{\alpha} = \beta$. There may not exist
an $\alpha$ for a given $\beta$ because some commutator relations in $G$ may be
trivial in $L(\phi)$. We choose transversals for each $s\in M$, say $\tau_s :
\phi_s/\partial\phi_s \rightarrow \phi_s$ and $\sigma_s :
\theta_s/\partial\theta_s \rightarrow \theta_s$. For each $s\in M$, let
$\mathcal{X}_s$ be the image of a generating set for $L_s(\phi)$ under $\tau_s$.
Define a \emph{partial lift} of $\beta$ to be the function $\alpha :
\bigsqcup_{s\in \mathcal{I}_{\phi}} \mathcal{X}_s \rightarrow H$ such that if
$x\in \mathcal{X}_s$, then $x\mapsto (\partial\phi_sx)^{\beta\sigma_s}$. If
$\alpha$ induces a group homomorphism, then such a homomorphism is called a
\emph{lift} of $\beta$. 

\begin{thmalpha}\label{thm:Main3}
   Suppose $G$ and $H$ are groups and $\alpha: G\rightarrow H$ is an
   isomorphism. If $\phi$ is a finite, inertia-free, faithful $(M, G)$-filter,
   then $\alpha$ is a lift of $\hat{\alpha}$.
\end{thmalpha}

If, in particular, $\phi$ is also a characteristic $(M, G)$-filter, for example
like in \cites{M:classical, M:efficient-filters, W:char}, then every
automorphism of $G$ is a lift of an $M$-graded Lie automorphism of $L(\phi)$. We
say an $M$-filter functor $\phi$ satisfies a property if for all groups $G$,
$\phi(G)$ satisfies that same property. With this, Theorem~\ref{thm:Main3}
implies Theorem~\ref{thm:Main0}.

\begin{remark}
   In the context of isomorphism testing of finite $p$-groups, one may want
   monoids, or semigroups, that are not as general as we allow here. Indeed,
   focusing on a specific class of monoids may reduce complexity of algorithms.
   For example, semigroups with the relation that, for some integer $n$, all
   sums of at least $n+1$ nontrivial elements are the same. For finite
   $p$-groups $G$, the integer $n$ could be the nilpotency class of $G$ or
   $\log_p|G|$ for example. Filters $\phi$ over these semigroups are
   inertia-free, so by Theorem~\ref{thm:Main}, there exists a surjection from
   $L(\phi)$ to $G$. In order to take advantage of Theorem~\ref{thm:Main3}, care
   is still needed to construct faithful filters. 
\end{remark}

\subsection{Overview}

Section~\ref{sec:background} details preliminary definitions and theorems needed
for the rest of the paper. We include examples of filters referenced in the
introduction. In Section~\ref{sec:monoids} we prove statements about conically
pre-ordered monoids that are used in Section~\ref{sec:inert}, where we provide a
general construction for filters for which there exists a surjection from the
Lie ring to the group---thereby proving Theorem~\ref{thm:Main}. In
Sections~\ref{sec:partially-ordered} and~\ref{sec:faithful-filters} we turn to
desireable combinatorial properties from lattices. We investigate structural
properties of faithful and inertia-free filters, and we prove
Theorems~\ref{thm:Main2} and~\ref{thm:Main3}. We close with some examples in
Section~\ref{sec:examples-section}.

\section{Preliminaries}\label{sec:background}

\subsection{Notation}\label{sec:assumptions}
For a set $\mathcal{X}$, we let $2^{\mathcal{X}}$ denote the power set of
$\mathcal{X}$. For a group $G$, let $\Nor(G)\subset 2^G$ denote the set of
normal subgroups of $G$. We denote the set of nonnegative integers by $\N$.

For $x,y\in G$, set $[x] = x$ and $[x,y] = x^{-1}y^{-1}xy$. For $x_1,\dots,
x_n\in G$, we recursively define $[x_1,\dots, x_n] = [[x_1,\dots, x_{n-1}],
x_n]$. We say the commutator $[x_1,\dots, x_n]$ has \emph{weight} $n$. For
$X,Y\subseteq G$, set $[X,Y] = \langle [x,y] ~|~ x\in X, y\in Y\rangle$; we
apply the same recursive formula for commutators of subsets of weight $n$. Let
$\gamma_1=G$ and for $i\geq 1$, set $\gamma_{i+1} = [\gamma_i, G]$. A nilpotent
group has class $c$ if $\gamma_c > \gamma_{c+1} = 1$.

A commutative monoid $\langle M,+,0\rangle$ is pre-ordered by a pre-order
$\preceq$ if $s\preceq t$ and $s'\preceq t'$ imply that $s+s'\preceq t+t'$.
Throughout, we will use $\preceq$ by default for the pre-order on $M$. For $s,t
\in M$, we let $s\prec t$ denote $s\preceq t$ and $s\ne t$ (in general $\prec$
is not transitive). An element $s\in M$ is a \emph{unit} if there exists $t\in
M$ such that $s+t=0$. For $s,t\in M$, we let $s\parallel t$ denote the case when
$s$ and $t$ are incomparable under $\preceq$, i.e.\ $s\not\preceq t$ and
$t\not\preceq s$. A subset $S\subseteq M$ is a $\preceq$-\emph{chain} if $S$ is
totally ordered with respect to $\preceq$. Similarly, $S\subseteq M$ is an
$\preceq$-\emph{antichain} if every distinct pair of elements of $S$ is
incomparable.

A partially ordered set $L$ is a \emph{lattice} if for all $x, y\in L$ both
$x\cap y\in L$ and $x\cup y\in L$. A partially ordered set $L$ is a
\emph{complete lattice} if for all $X\subseteq L$ both $\bigcap_{x\in X}x\in L$
and $\bigcup_{x\in X}x\in L$. All of our lattices are sublattices of either
$2^{\mathcal{X}}$ or $\Nor(G)$. Therefore, $\cap$ and $\cup$ are understood to
be intersection and either set or subgroup union, respectively.

\subsection{Monoids and pre-orders}

An important pre-order on monoids is the \emph{algebraic pre-order} denoted
$\preceq_+$ where for $s,t\in M$,
\begin{align*} 
   s \preceq_+ t \implies \exists u\in M, s+u=t.
\end{align*}
Every commutative monoid is pre-ordered by $\preceq_+$. Another pre-order that
we will use in examples later is the lexicographical (abbreviated lex)
pre-order. Suppose $(M,\leq)$ and $(N,\preceq)$ are two pre-ordered commutative
monoids. The \emph{lex order} of $M\times N$, denoted $\leq_{\ell}$, is defined
as follows. For all $(m,n),(m',n')\in M\times N$, 
\begin{align*} 
   (m,n) \leq_{\ell} (m',n') \iff (m < m') \text{ or } (m=m' \text{ and } n\preceq n').
\end{align*}

It is possible to describe all cyclic monoids up to isomorphism;
cf.~\cite{Grillet:book}*{Proposition~5.8}. Let $r,s\in\N$ with $s\geq 1$. Define
a congruence $\sim$ on $\N$ where $i,j\in \N$,
\begin{align*} 
   i\sim j \Longleftrightarrow \left\{ \begin{array}{ll} 
      i\equiv j\; (\text{mod }s) & \text{if }i,j\geq r, \\
      i=j & \text{otherwise}. 
   \end{array}\right.
\end{align*}
Define $C_{r,s}= \N /\!\!\sim$ with addition induced from $\N$, and note that
$|C_{r,s}|=r+s$. 

\subsection{Examples of filters}\label{sec:examples}

We provide some explicit examples of properties we want to avoid. To emphasize
that these properties are inherently concerned with the monoid, we use a similar
group in all these examples. For a ring $R$, we denote the Heisenberg group over
$R$ by
\[ 
   H(R) = \left\{ \begin{bmatrix} 
      1 & a & c \\ 
      & 1 & b \\ 
      & & 1 
   \end{bmatrix}\;\middle|\; a,b,c\in R\right\}. 
\]
We plot all of the filters in this section in Figure~\ref{fig:ex1}.

\begin{ex}\label{ex:trivial}
   Here we construct a filter such that the associated Lie ring is trivial.
   Let $G=H(\mathbb{Z})$, and let $M=(\N^2,\preceq_\ell)$, 
   ordered by the lex ordering. For $s\in M$,          
   \[ 
      \phi_s = \left\{ \begin{array}{ll} 
         G & s\prec_\ell (2,0), \\ 
         G' & (2,0)\preceq_\ell s \prec_\ell (3,0), \\ 
         1 & (3,0) \preceq_\ell s. 
      \end{array}\right. 
   \]
   For all $s\in M$, $\phi_{s} = \phi_{s+(0,1)}$. Therefore, for all $s\in M$,
   $\partial\phi_s=\phi_s$, so $L(\phi)=0$. In fact, $\phi$ is the resulting
   filter when applied to the generation formula in~\cite{W:char}*{Theorem~3.3},
   with $X=\{s\in M\mid s\prec_\ell (2,0)\}$ and $\pi:X\rightarrow \{G\}$ the
   constant function.
\end{ex}

\begin{ex}\label{ex:infinite}
   We construct a family of filters for a fixed finite group $G$ whose associated Lie algebras are of arbitrarily large dimension.
   Let $K$ be a finite field, $G=H(K)$ and $M=(\N^2, \preceq_+)$. Fix $n\geq 2$. For $s=(i,j)\in M$ define
   \begin{align*}
      \phi_s &= \left\{ \begin{array}{ll} 
         G & i+j \leq 1, \\
         G' & 2 \leq i+j \leq n,\\
         1 & \text{otherwise}.
      \end{array}\right.
   \end{align*}
   The boundary filter is given by 
   \begin{align*}
      \partial \phi_s &= \left\{ \begin{array}{ll} 
         G & i = j = 0, \\
         G' & 1 \leq i+j \leq n-1,\\
         1 & \text{otherwise},
      \end{array}\right.
   \end{align*}
   where $s=(i,j)\in M$. It follows that $L(\phi)$ is a $K$-algebra, and the
   dimension of $L(\phi)$ over $K$ is $n+5$. 
\end{ex}

\begin{ex}\label{ex:just-bijection}
   We combine aspects of the Examples~\ref{ex:trivial} and~\ref{ex:infinite} and
   construct a filter where the Lie ring and group are in bijection as sets.
   However, no pre-image of a pcgs for the Lie ring induces a pcgs for the
   group. 

   Let $G=H(K)$ for a finite field $K$ and $M=(\N^2,\preceq_+)$. For $s=(i,j)\in\N^2$, define 
   \begin{align*}
      \phi_s &= \left\{ \begin{array}{ll}
         G & i = 0, \\
         G' & i = 1\text{ or } i+j\leq 4, \\
         1 & \text{otherwise}.
      \end{array}\right.
   \end{align*}
   The border set is $\mathcal{I}_\phi = \{ (4,0), (3, 1), (2, 2) \}$, so as $K$-vector spaces $L(\phi) \cong K^3$. Therefore, $G$ and $L(\phi)$ are in bijection, but $L(\phi) = (G')^3$ is an abelian Lie algebra. In particular, every pre-image of a pcgs for $L(\phi)$ will be contained in the center of $G$. 
\end{ex}

\begin{ex}\label{ex:not-progressive}
   We construct an example that conspicuously hides subgroups.
   Let $M = C_{3,1}\times C_{1, 1}$ be the monoid with pre-order given by the algebraic order $\preceq_+$. 
   Let $G=H(\mathbb{Z})$, and suppose $K\triangleleft G'$ has index $2$ in $G'$. Define an $(M, G)$-filter $\phi$ such that, for $s=(i,j)\in M$,
   \begin{align*}
      \phi_s &=\left\{ \begin{array}{ll}
         G & i = j = 0,\\
         \gamma_i(G) & i > j = 0,\\
         K & i\leq 2, j=1, \\
         1 & i=3, j=1.
      \end{array}\right.
   \end{align*}
   The boundary filter is then
   \begin{align*}
      \partial \phi_s &=\left\{ \begin{array}{ll}
         G & i = j = 0,\\
         G' & (i, j) = (1,0),\\
         1 & i=3, \\
         K & \text{otherwise}.
      \end{array}\right.
   \end{align*}
   Therefore, as abelian groups $L(\phi) \cong \mathbb{Z}^2 \oplus \mathbb{Z}/2\mathbb{Z}$. This filter also serves as an example of a filter for which our construction in Section~\ref{sec:inert} cannot be applied. Since the order on $M$ is a partial order, Theorem~\ref{thm:Main} can still be applied to this filter, but we need to change the monoid.
\end{ex}

\begin{figure}[ht]
   \begin{subfigure}[b]{0.48\textwidth}
      \centering
      \begin{tikzpicture}
         \pgfmathsetmacro{\myscale}{0.6}
         \node (ptx) at (-0.328,-0.2) {};
         \node (pty) at (-0.2,-0.328) {};
         \node (x) at (0.4 + 3*\myscale,-0.2) {};
         \node (y) at (-0.2, 0.45 + 2*\myscale) {};
         \draw[->] (ptx) -- (x);
         \draw[->] (pty) -- (y);
      
         \node (x0) at (0,-0.5) {0};
         \node (x1) at (\myscale,-0.5) {1};
         \node (x2) at (2*\myscale,-0.5) {2};
         \node (x3) at (3*\myscale,-0.5) {3};
         \node (y0) at (-0.5,0) {0};
         \node (y1) at (-0.5,\myscale) {1};
         \node (y2) at (-0.5,2*\myscale) {2};
         
         \node (G) at (0,0) {$G$};
         \node (H) at (\myscale,0) {$G$};
         \node (G3) at (2*\myscale,0) {$G'$};
         \node (G4) at (3*\myscale,0) {$1$};
         \node (K) at (0,\myscale) {$G$};
         \node (G2) at (\myscale,\myscale) {$G$};
         \node (G32) at (2*\myscale,\myscale) {$G'$};
         \node (G42) at (3*\myscale,\myscale) {$1$};
         \node (L) at (0,2*\myscale) {$G$};
         \node (G33) at (\myscale,2*\myscale) {$G$};
         \node (G43) at (2*\myscale,2*\myscale) {$G'$};
         \node (G44) at (3*\myscale,2*\myscale) {$1$};
      \end{tikzpicture}
      \caption{The filter from Example~\ref{ex:trivial}.}
      \label{fig:ex-trivial}
   \end{subfigure}\hfill%
   \begin{subfigure}[b]{0.48\textwidth}
      \centering
         \begin{tikzpicture}
            \pgfmathsetmacro{\myscale}{0.6}
            \node (ptx) at (-0.328,-0.2) {};
            \node (pty) at (-0.2,-0.328) {};
            \node (x) at (0.4 + 4*\myscale,-0.2) {};
            \node (y) at (-0.2, 0.45 + 3*\myscale) {};
            \draw[->] (ptx) -- (x);
            \draw[->] (pty) -- (y);
         
            \node (x0) at (0,-0.5) {0};
            \node (x1) at (\myscale,-0.5) {1};
            \node (x2) at (2*\myscale,-0.5) {2};
            \node (x3) at (3*\myscale,-0.5) {3};
            \node (x4) at (4*\myscale,-0.5) {4};
            \node (y0) at (-0.5,0) {0};
            \node (y1) at (-0.5,\myscale) {1};
            \node (y2) at (-0.5,2*\myscale) {2};
            \node (y3) at (-0.5,3*\myscale) {3};
            
            \node (G) at (0,0) {$G$};
            \node (H) at (\myscale,0) {$G$};
            \node (G3) at (2*\myscale,0) {$G'$};
            \node (G4) at (3*\myscale,0) {$G'$};
            \node (G40) at (4*\myscale,0) {$1$};
            \node (K) at (0,\myscale) {$G$};
            \node (G2) at (\myscale,\myscale) {$G'$};
            \node (G32) at (2*\myscale,\myscale) {$G'$};
            \node (G42) at (3*\myscale,\myscale) {$1$};
            \node (G41) at (4*\myscale,\myscale) {$1$};
            \node (L) at (0,2*\myscale) {$G'$};
            \node (G33) at (\myscale,2*\myscale) {$G'$};
            \node (G43) at (2*\myscale,2*\myscale) {$1$};
            \node (G44) at (3*\myscale,2*\myscale) {$1$};
            \node (G43) at (4*\myscale,2*\myscale) {$1$};
            \node (14) at (0,3*\myscale) {$G'$};
            \node (15) at (\myscale,3*\myscale) {$1$};
            \node (16) at (2*\myscale,3*\myscale) {$1$};
            \node (17) at (3*\myscale,3*\myscale) {$1$};
            \node (18) at (4*\myscale,3*\myscale) {$1$};
         \end{tikzpicture}
         \caption{The filter from Example~\ref{ex:infinite}, for $n=3$.}
         \label{fig:ex-infinite}
   \end{subfigure}
   \begin{subfigure}[b]{0.48\textwidth}
      \centering
         \begin{tikzpicture}
            \pgfmathsetmacro{\myscale}{0.5}
            \node (ptx) at (-0.328,-0.2) {};
            \node (pty) at (-0.2,-0.328) {};
            \node (x) at (0.4 + 5*\myscale,-0.2) {};
            \node (y) at (-0.2, 0.45 + 4*\myscale) {};
            \draw[->] (ptx) -- (x);
            \draw[->] (pty) -- (y);
         
            \node (x0) at (0,-0.5) {0};
            \node (x1) at (\myscale,-0.5) {1};
            \node (x2) at (2*\myscale,-0.5) {2};
            \node (x3) at (3*\myscale,-0.5) {3};
            \node (x4) at (4*\myscale,-0.5) {4};
            \node (x5) at (5*\myscale,-0.5) {5};
            \node (y0) at (-0.5,0) {0};
            \node (y1) at (-0.5,\myscale) {1};
            \node (y2) at (-0.5,2*\myscale) {2};
            \node (y3) at (-0.5,3*\myscale) {3};
            \node (y4) at (-0.5,4*\myscale) {4};
            
            \node (G) at (0,0) {$G$};
            \node (H) at (\myscale,0) {$G'$};
            \node (G3) at (2*\myscale,0) {$G'$};
            \node (G4) at (3*\myscale,0) {$G'$};
            \node (G40) at (4*\myscale,0) {$G'$};
            \node (G40) at (5*\myscale,0) {$1$};
            \node (K) at (0,\myscale) {$G$};
            \node (G2) at (\myscale,\myscale) {$G'$};
            \node (G32) at (2*\myscale,\myscale) {$G'$};
            \node (G42) at (3*\myscale,\myscale) {$G'$};
            \node (G41) at (4*\myscale,\myscale) {$1$};
            \node (G41) at (5*\myscale,\myscale) {$1$};
            \node (L) at (0,2*\myscale) {$G$};
            \node (G33) at (\myscale,2*\myscale) {$G'$};
            \node (G43) at (2*\myscale,2*\myscale) {$G'$};
            \node (G44) at (3*\myscale,2*\myscale) {$1$};
            \node (G43) at (4*\myscale,2*\myscale) {$1$};
            \node (G43) at (5*\myscale,2*\myscale) {$1$};
            \node (14) at (0,3*\myscale) {$G$};
            \node (15) at (\myscale,3*\myscale) {$G'$};
            \node (16) at (2*\myscale,3*\myscale) {$1$};
            \node (17) at (3*\myscale,3*\myscale) {$1$};
            \node (18) at (4*\myscale,3*\myscale) {$1$};
            \node (18) at (5*\myscale,3*\myscale) {$1$};
            \node (14) at (0,4*\myscale) {$G$};
            \node (15) at (\myscale,4*\myscale) {$G'$};
            \node (16) at (2*\myscale,4*\myscale) {$1$};
            \node (17) at (3*\myscale,4*\myscale) {$1$};
            \node (18) at (4*\myscale,4*\myscale) {$1$};
            \node (18) at (5*\myscale,4*\myscale) {$1$};
         \end{tikzpicture}
         \caption{The filter from Example~\ref{ex:just-bijection}.}
         \label{fig:ex-bijection}
   \end{subfigure}\hfill%
   \begin{subfigure}[b]{0.48\textwidth}
      \centering
      \begin{tikzpicture}
         \pgfmathsetmacro{\myscale}{0.6}
         \node (ptx) at (-0.328,-0.2) {};
         \node (pty) at (-0.2,-0.328) {};
         \node (x) at (0.4 + 3*\myscale,-0.2) {};
         \node (y) at (-0.2, 0.45 + 1*\myscale) {};
         \draw[-] (ptx) -- (x);
         \draw[-] (pty) -- (y);
      
         \node (x0) at (0,-0.5) {0};
         \node (x1) at (\myscale,-0.5) {1};
         \node (x2) at (2*\myscale,-0.5) {2};
         \node (x3) at (3*\myscale,-0.5) {3};
         \node (y0) at (-0.5,0) {0};
         \node (y1) at (-0.5,\myscale) {1};
         
         \node (G) at (0,0) {$G$};
         \node (H) at (\myscale,0) {$G$};
         \node (G3) at (2*\myscale,0) {$G'$};
         \node (G4) at (3*\myscale,0) {$1$};
         \node (K) at (0,\myscale) {$K$};
         \node (G2) at (\myscale,\myscale) {$K$};
         \node (G32) at (2*\myscale,\myscale) {$K$};
         \node (G42) at (3*\myscale,\myscale) {$1$};
      \end{tikzpicture}
      \caption{The filter from Example~\ref{ex:not-progressive}.}
      \label{fig:ex-progressive}
   \end{subfigure}
   \caption{The plots of every filter in Section~\ref{sec:examples}.}
   \label{fig:ex1}
\end{figure}

\section{Monoids associated to filters}\label{sec:monoids}

The assumption that $0$ is minimal is important for filters which also restricts
the possible monoids we consider. Recall, $s\in M$ is a unit of $M$ if there
exists $t\in M$ such that $s+t=0$. A commutative monoid $M$ is \emph{conical} if
$s+t=0$ implies $s=t=0$, for all $s,t\in M$. 

\begin{lem}\label{lem:conical}
   Suppose $M$ is a pre-ordered monoid. If $0$ is the minimal element of $M$,
   then 
   \begin{enumerate}
      \item[$(i)$] $M$ is conical,
      \item[$(ii)$] the only unit of $M$ is $0$, and 
      \item[$(iii)$] for all $s, t\in M$, if $s\preceq_+t$, then $s\preceq t$. 
   \end{enumerate} 
\end{lem}

\begin{proof}
   All three statements follow from the formula: $s = s + 0 \preceq s + u = t$. 
\end{proof}

We will need to show that in a conically pre-ordered monoid $M$, there is no
infinite set of $\preceq_+$-antichains. To do this, we will reduce the statement
to the following lemma, which follows from the pigeonhole principle.

\begin{lem}\label{lem:antichain}
   Let $m\in\N$. If $S\subseteq \N^m$ such that $S$ is an $\preceq_+$-antichain,
   then $|S|<\infty$. 
\end{lem}

\begin{proof}
   The statement is clear if $m=1$, so we assume $m\geq 2$. If $s = (s_1, \dots,
   s_m)\in S$, then for each $k$, define $R_k(s) = \{ (t_1,\dots, t_m)\in \N^m
   \mid t_k < s_k \}$. Because $S$ is an $\preceq_+$-antichain, $S\setminus\{
   s\} \subseteq \bigcup_{k=1}^m R_k(s)$. If for each $k$, the intersection $I_k
   = (S\setminus\{s\}) \cap R_k(s)$ is finite, then $S$ is finite and we are
   done. Otherwise, there exists $k$ such that $I_k$ is infinite. Again, by the
   pigeonhole principle, there exists $r_k < s_k$ such that $\{(t_1,\dots, t_m)
   \in I_k \mid t_k=r_k \}$ is infinite. This implies the existence of an
   infinite $\preceq_+$-antichain in $\N^{m-1}$. Since this does not hold for
   $m=1$, $S$ is finite. 
\end{proof}  

\begin{definition}\label{def:unit-cancel}
   An element $s\in M$ is \emph{unit-cancellative} if $s+t = s$ implies that $t$
   is a unit. We say $s\in M$ is a \emph{sink} if it is not unit-cancellative.
\end{definition}

\begin{definition}\label{def:atom}
   An element $s\in M$ is an \emph{atom} if $s = t + u$ implies that either
   $t=0$ or $u=0$. 
\end{definition}

The next proposition is critical to a number of statements about the structure
of $(M,G)$-filters. We prove a finiteness property on the set of
unit-cancellative elements of conically pre-ordered monoids. The property is
well-studied in the context of non-unique factorization in rings of integers of
number fields, see \cite{GH:Book}*{Sections~1.1 \&~1.5}.

\begin{prop}\label{prop:unit-cancel}
   If $M$ is a conically pre-ordered monoid, then there exists a finite set of
   unit-cancellative atoms $\mathcal{U}$, generating all unit-cancellative
   elements of $M$.
\end{prop}

\begin{proof}
   First, we prove that if $s\in M$ is unit-cancellative and $s=t+u$ for $t,u\in M$, then $t$ and $u$ are unit-cancellative. Suppose $t+t'=t$ and $u+u'=u$. Then, 
   \[s = t+ u = t+t'+u+u' = s + t'+u'. \]
   Since the only unit of $M$ is $0$ by Lemma~\ref{lem:conical}, it follows that $t'+u'=0$. Moreover, since $M$ is conical, $t'=u'=0$. Thus, the unit-cancellative elements of $M$ decompose as a (possibly trivial) sum of unit-cancellative elements. 

   Since $M$ is finitely generated, let $\mathcal{S}$ be a minimal generating set for $M$, and suppose $\mathcal{U}\subseteq \mathcal{S}$ is the subset of unit-cancellative elements of $\mathcal{S}$. From the argument above, if $s\in M$ is unit-cancellative, then $s$ is contained in the monoid generated by $\mathcal{U}$. 

   Now we show that for a fixed $u\in \mathcal{U}$, if $s+t=u$, then either $s=0$ or $t=0$. By above, if $s+t=u$, then both $s$ and $t$ are unit-cancellative. Since the only unit of $M$ is $0$, by Lemma~\ref{lem:conical}, there exist $(a_v)_{v\in \mathcal{U}}, (b_v)_{v\in\mathcal{U}} \in\N^{|\mathcal{U}|}$ such that 
   \begin{align*}
      s &= \sum_{v\in\mathcal{U}} a_v\cdot v, & t &= \sum_{v\in\mathcal{U}} b_v\cdot v.
   \end{align*}
   By minimality of $\mathcal{S}$ and $\mathcal{U}$, either $s$ or $t$ is not generated by $\mathcal{U}\setminus \{u\}$, so in particular either $a_u$  or $b_u$ is nonzero. Suppose $b_u\ne 0$. Therefore, $b_u-1\in\N$, and 
   \[ u = s + t = \sum_{v\in\mathcal{U}}(a_v + b_v)v = \sum_{v\in\mathcal{U}\setminus \{u\}} (a_v+b_v)v + (a_u + b_u - 1)u + u,\]
   this implies that $a_v+b_v = 0$ for all $v\in\mathcal{U}\setminus \{u\}$ and $a_u+b_u-1=0$. Hence, $s=0$.
\end{proof}

We extend Proposition~\ref{prop:unit-cancel} a little further. Because there
exists a set of unit-cancellative atoms, there are only finitely many distinct
ways to express a unit-cancellative element as a sum of non-zero elements. 

\begin{cor}\label{cor:finite-express}
   If $M$ is conically pre-ordered and $s\in M$ is unit-cancellative, then the
   following set is finite
   \[ \Zmon(s) = \left\{ (t_1,\dots, t_k) \in (M\setminus \{0\})^k ~\middle|~ k\in\N, \sum_{i=1}^k t_i = s \right\}. \]
\end{cor}

\begin{proof}
   By Proposition~\ref{prop:unit-cancel}, it follows that $\Zmon(s)$ is finite if, and only if, 
   \[ \Nmon(s) = \left\{ (n_u)_{u\in\mathcal{U}}\in \N^{|\mathcal{U}|} ~\middle|~ \sum_{u\in\mathcal{U}} n_u\cdot u = s \right\} \]
   is finite. Suppose $\Nmon(s)$ is infinite. By Lemma~\ref{lem:antichain}, since $\mathcal{U}$ is finite, there exist distinct $(n_u)_{u\in\mathcal{U}},(n_u')_{u\in\mathcal{U}} \in\Nmon(s)$ such that $n_u\leq n_u'$, for all $u\in\mathcal{U}$.  This implies that $d_u = n_u' - n_u\in\N$ and for some $u\in\mathcal{U}$, $d_u \geq 1$. But since $s$ is unit-cancellative,
   \[ \sum_{u\in\mathcal{U}} n_u\cdot u = s = s + \sum_{u\in\mathcal{U}} d_u\cdot u = \sum_{u\in\mathcal{U}} n_u'\cdot u \]
   Thus, for all $u\in\mathcal{U}$, $d_u = 0$, a contradiction. Hence, $\Nmon(s)$ and $\Zmon(s)$ are finite.
\end{proof}

We are in a position to prove a critical statement about chains of
unit-cancellative elements in conically pre-ordered monoids that will be used in
the next section.

\begin{cor}\label{cor:unit-cancel-dcc-anti}
   If $M$ is a conically pre-ordered monoid, then 
   \begin{enumerate}
      \item[$(i)$] every descending $\preceq_+$-chain of unit-cancellative
      elements stabilizes, and 
      \item[$(ii)$] every $\preceq_+$-antichain of unit-cancellative elements is
      finite. 
   \end{enumerate}
\end{cor}

\begin{proof}
   For $(i)$, if $s_0\succeq_+ s_1\succeq_+\cdots$, then there exists $t_i\in M$
   such that $s_{i-1} = s_{i} + t_{i}$ for $i\geq 1$. Therefore, $s_0 = s_k +
   \sum_{j=1}^k t_j$, for every $k\geq 1$. By
   Corollary~\ref{cor:finite-express}, all but finitely many $t_i=0$. 

   For $(ii)$, suppose $\{s_i\}_{i\geq 1}\subseteq M$ is an
   $\preceq_+$-antichain. Let $\mathcal{U}\subseteq M$ be the unit-cancellative
   atoms of $M$. An infinite $\preceq_+$-antichain of unit-cancellative elements
   in $M$ is equivalent to an infinite $\preceq_+$-antichain in
   $\N^{|\mathcal{U}|}$. Since $\mathcal{U}$ is finite, this is impossible by
   Lemma~\ref{lem:antichain}. 
\end{proof}

\begin{remark}
   The Infinite Ramsey Theorem along with
   Corollary~\ref{cor:unit-cancel-dcc-anti} imply the following statement about
   $(M,G)$-filters $\phi$. The set $\im(\phi)$ is finite if, and only if,
   $\im(\phi)$ satisfies the descending chain condition.
\end{remark}

\section{Inert subgroups of filters}\label{sec:inert}

In Example~\ref{ex:trivial}, the filter $\phi$ has the property that
$L(\phi)=0$. This is an extreme example, but this illustrates a property we want
to repair. Recall that for a conical monoid $M$, $s\in M$ is unit-cancellative
if $s+t=s$ implies $t=0$, cf.\ Lemma~\ref{lem:conical} and
Definition~\ref{def:unit-cancel}. 

\begin{definition}
   An $(M, G)$-filter $\phi$ is \emph{progressive} if  $s\in M$ is
   a sink implies that $\phi_s = \langle 1\rangle$.
\end{definition}

From Section~\ref{sec:examples}, Examples~\ref{ex:trivial},~\ref{ex:infinite},
and~\ref{ex:just-bijection} are progressive filters, but
Example~\ref{ex:not-progressive} is not progressive. By definition, $\langle
1\rangle$ is not an inert subgroup of an $(M,G)$-filter $\phi$. We
now state our main theorems for this section; the combination of which implies
Theorem~\ref{thm:Main}.

\begin{thm}\label{thm:general-progressive}
   Suppose $\phi$ is a finite $(M, G)$-filter. If either $\phi$ is progressive
   or $\preceq$ is a partial order on $M$, then there exists a conically
   pre-ordered monoid $M'$ and a finite, inertia-free $(M',G)$-filter $\theta$
   such that $\im(\phi)\subseteq\im\left(\theta\right)$.
\end{thm}

\begin{thm}\label{thm:surjection}
   Suppose $\phi$ is a finite $(M,G)$-filter. If $\phi$ is inertia-free, then
   there exists a graded generating set $\mathcal{Y}$ for the abelian group
   $L(\phi)$ and a surjection $\pi_{\mathcal{Y}}: L(\phi) \rightarrow G$ of sets
   such that $\pi_{\mathcal{Y}}(\mathcal{Y})$ generates $G$.
\end{thm}

The next proposition is critical to working with inertia in filters. It gives us
a useful characterization of inert subgroups in this section and, in the coming
sections, gives us a foothold to apply Noetherian induction on the subgroups in
$\im(\phi)$. 

Recall the border set of an $(M, G)$-filter $\phi$, given
in~\eqref{eqn:relevant-set}, is denoted by $\mathcal{I}_{\phi}$. We characterize
the inertia-free filters by showing that $\phi_s\in \mathfrak{B}_n$, with
$\partial\phi_s\ne \langle 1\rangle$, implies that there exists a subset
$B\subseteq \mathfrak{B}_{n-1}$ and $I\subseteq\mathcal{I}_\phi$ such that
$B=\{ \phi_t\mid t\in I_s\}$. We do this constructively. If this property does
not hold for a particular $H\in B$, then we choose a better subset $C$ where $H$
is replaced by a subset of subgroups that generate $H$. If no such choice is
available, then $\phi_s$ must be an inert subgroup of $\phi$. We prove this by
induction, moving up the $\mathfrak{B}$-chain.

\begin{prop}\label{prop:inert}
   Suppose $\phi$ is a finite $(M, G)$-filter. For all $s\in M$,
   there exists $I_s\subseteq \mathcal{I}_\phi$ such that
   $\partial\phi_s=\langle \phi_t\mid t\in I_s\rangle$ if, and only if,
   $\im(\phi)=\mathfrak{B}$.
\end{prop}

\begin{proof}
   Suppose first, via contradiction, that $\phi_s\notin \mathfrak{B}$. Since
   $\phi$ is finite, we assume $\phi_s$ is minimal. By the assumption, there
   exists $I_s\subseteq \mathcal{I}_\phi$ such that $\partial\phi_s=\langle
   \phi_t\mid t\in I_s\rangle$. If for every $t\in I_s$, $\phi_t\in
   \mathfrak{B}$, then $\phi_s\in\mathfrak{B}$. Therefore, since
   $\phi_s\notin\mathfrak{B}$, we choose $t\in I_s$ such that
   $\phi_{t}\notin\mathfrak{B}$. Since $t\in\mathcal{I}_\phi$, $\phi_{t} >
   \partial\phi_{t}$, so 
   \[ 
      \phi_s\geq\partial\phi_s\geq \phi_{t}>\partial\phi_{t}.
   \]
   By assumption, there exists $I_{t}\subseteq\mathcal{I}_\phi$ such that
   $\partial\phi_{t}=\langle \phi_u \mid u\in I_{t}\rangle$. By minimality of
   $\phi_s$, for each $u\in I_t$, $\phi_u\in \mathfrak{B}$. This implies that
   $\phi_t\in\mathfrak{B}$, which is a contradiction. Therefore,
   $\phi_s\in\mathfrak{B}$ for all $s\in M$.  

   Conversely, suppose $\im(\phi)=\mathfrak{B}$. If there exists $\phi_s\in
   \im(\phi)$ such that for all $I_s\subseteq \mathcal{I}_\phi$ where
   $\partial\phi_s \ne \langle \phi_t\mid t\in I_s\rangle$, then there exists a
   minimal such $\phi_s\in \im(\phi)$ as $\phi$ is finite. We fix this $\phi_s$.
   Since $\phi_s\in\mathfrak{B}$, it follows that there exists
   $B\subset\bigcup_{i\geq 0}\mathfrak{B}_i$ such that $\partial\phi_s=\langle
   B\rangle$, and without loss of generality, we assume that $B$ is minimal.
   Therefore, we have two cases. Because $B$ is minimal, either $B=\{\phi_s\}$
   or $\phi_s\notin B$. 

   First, suppose $B=\{\phi_s\}$, and suppose that $B\subseteq \mathfrak{B}_m$,
   where $m$ is minimal. Therefore, there exists $C\subseteq \mathfrak{B}_{m-1}$
   such that $\partial\phi_s = \langle C\rangle$. By the minimality of $m$,
   $\phi_s\notin C$. Hence, without loss of generality, we may assume
   $\phi_s\notin B$. By minimality of $\phi_s$, every $\phi_t\in B$ satisfies
   $(i)$. Therefore, there exists $I_t\subseteq \mathcal{I}_\phi$ such that
   $\partial\phi_t = \langle \phi_u \mid u\in I_t\rangle$. If
   $t\notin\mathcal{I}_\phi$, then $\partial\phi_t=\phi_t$. Therefore, we can
   replace $B$ with $C = (B \setminus \{\phi_t\} )\cup \{ \phi_u \mid u\in
   I_t\}$ and $\partial\phi_s=\langle C\rangle$. Applying this to every
   $\phi_t\in B$ such that $t\notin\mathcal{I}_\phi$ yields a new set that
   satisfies $(i)$.  Therefore, this holds for all $s\in M$. 
\end{proof}

\subsection{Refreshing filters}\label{sec:refreshing-filters}

In this subsection we show how to remove the inertia from finite, progressive
$(M, G)$-filters $\phi$. To do this, we attempt to make every nontrivial
subgroup have finite support on $M$. We accomplish this by constructing a new
filter $\theta$ with a two step process which can be seen as forcing the last
two filter properties from Definition~\ref{def:filter}. The first step
guarantees $[\theta_s,\theta_t]\leq \theta_{s+t}$, and the second step forces
the order-reversing property: $s\preceq t$ implies $\theta_s\geq \theta_t$.

Throughout the remainder of this subsection, we fix a minimal set of generators
for $M$, denoted by $\mathcal{S}$. For $E\subseteq M$, define a set of
restricted partitions of $M$ as follows. If $s\in M$, then define the set of
\emph{$E$-excluded partitions} of $s$ in $M$ to be
\begin{align*}\label{eqn:restricted-part} 
   \respart(s) &= \left\{ (r_1,\dots,r_k) ~\middle|~ k\in\N, r_i\in (M\setminus E)\cup\mathcal{S},\; r_1+\cdots+r_k=s \right\}. 
\end{align*}
For $E\subseteq M$, we define a function, $\nu=\nu(E)$, from $M$ into $\Nor(G)$.
For each $s\in M$, set
\begin{equation}\label{eqn:nu-subgroup}
   \nu_s = \prod_{{\bf r}\in \respart(s)} [\phi_{{\bf r}}],
\end{equation}
where $[\phi_{{\bf r}}]=[\phi_{r_1},\dots,\phi_{r_k}]$. Observe that for every
$s\in M$, the set $\respart(s)\ne\emptyset$ because $M=\langle
\mathcal{S}\rangle$. Although this definition depends on $\mathcal{S}$, we do
not explicitly mention $\mathcal{S}$ as it is assumed to be fixed. The next
lemma shows that $\nu$ and $\phi$ are equal on $(M\setminus E)\cup \mathcal{S}$.

\begin{lem}\label{lem:nu-leq}
   Suppose $\phi$ is an $(M, G)$-filter. If $E\subseteq M$ and $\nu=\nu(E)$ is
   defined as in~\eqref{eqn:nu-subgroup}, then for all $s\in M$, $\nu_s\leq
   \phi_s$. If $s\in (M\setminus E)\cup \mathcal{S}$, then equality holds. 
\end{lem}

\begin{proof}
   For $s\in M$, let $\textbf{r}\in \respart(s)$. Since $\phi$ is an $(M,
   G)$-filter, $[\phi_{\textbf{r}}]\leq \phi_s$. Therefore, $\nu_s\leq\phi_s$,
   for all $s\in M$. If $s\in (M\setminus E)\cup\mathcal{S}$, then $\textbf{r} =
   (s)\in \respart(s)$. Hence, $\nu_s\geq [\phi_{\textbf{r}}] = \phi_s$. Thus,
   the lemma follows. 
\end{proof}

Lemma~\ref{lem:nu-leq} hints at the reason why $\nu$ might fail to be an $(M,
G)$-filter: we cannot guarantee that $\nu$ is order-reversing. We include
another layer to our construction to get an $(M,G)$-filter. Define a function
$\widetilde{\nu}$ from $M$ into $\Nor(G)$, such that 
\begin{equation}\label{eqn:nu-filter} 
   \widetilde{\nu}_s = \prod_{s\preceq t}\nu_t = \prod_{s\preceq t}\left( \prod_{{\bf r}\in  \respart(t)} [\phi_{{\bf r}}]\right).
\end{equation}

At this point, we pause to give a map of the remainder of the subsection. The
first major destination is Proposition~\ref{prop:hat-filter}, which proves that
$\widetilde{\nu}$ is an $(M, G)$-filter with the properties we want. The coming
four lemmas work towards this goal. After Proposition~\ref{prop:hat-filter}, we
journey towards inertia-free. Here, we leverage Section~\ref{sec:monoids} to
show that, with a small change to $\widetilde{\nu}$, we get an inertia-free
filter. This involves investigating sequences in $M$. 

The next lemma is a useful extension of Lemma~\ref{lem:nu-leq}.

\begin{lem}\label{lem:hat-order}
   Suppose $\phi$ is an $(M, G)$-filter, $E\subseteq M$, $\nu=\nu(E)$, and
   $\widetilde{\nu}=\widetilde{\nu}(E)$. For $s\in (M\setminus
   E)\cup\mathcal{S}$, $\widetilde{\nu}_s = \phi_s$. 
\end{lem}

\begin{proof}
   By Lemma~\ref{lem:nu-leq}, $\nu_s = \phi_s$ for $s\in (M\setminus
   E)\cup\mathcal{S}$. Since $\phi$ is an $(M, G)$-filter, $s\preceq t$ implies
   $\phi_s\geq \phi_t$. Since $\nu_t\leq \phi_t$ for all $t\in M$, 
   \[ 
      \phi_s \leq \nu_s\prod_{s\prec t}\nu_t = \prod_{s\preceq t} \nu_t \leq \prod_{s\preceq t} \phi_t \leq \phi_s. 
   \] 
   By~\eqref{eqn:nu-filter}, $\phi_s\leq \widetilde{\nu}_s  = \prod_{s\preceq
   t}\nu_t \leq \phi_s$. 
\end{proof}

The difficult part of showing $[\widetilde{\nu}_s, \widetilde{\nu}_t]\leq
\widetilde{\nu}_{s+t}$, for $s,t\in M$, is showing this identity holds for
$\nu$, which is what we do first. The next lemma accomplishes this by applying
the Three Subgroups Lemma, cf.~\cite{Robinson:book}*{{\bf 5.1.10}}.

\begin{lem}\label{lem:nu-filter}
   Suppose $\phi$ is an $(M, G)$-filter. Let $E\subseteq M$ and $\nu=\nu(E)$. If
   $s,t\in M$, then $\left[\nu_s,\nu_t\right]\leq \nu_{s+t}$.
\end{lem}

\begin{proof}
   If $s+t\notin E$, then by Lemma~\ref{lem:nu-leq}, $\nu_s\leq\phi_s$,
   $\nu_t\leq\phi_t$, and $\nu_{s+t}=\phi_{s+t}$. Since $\phi$ is an $(M,
   G)$-filter, it follows that 
   \[ 
      [\nu_s,\nu_t]\leq [\phi_s,\phi_t]\leq \phi_{s+t}=\nu_{s+t}. 
   \]

   Now consider the case when $s+t\in E$. For $\textbf{s}\in \respart(s)$ and
   $\textbf{t}\in \respart(t)$, we prove 
   \begin{align}\label{eqn:nu-filter-reduced}
      [[\phi_{\textbf{s}}], [\phi_{\textbf{t}}]]\leq \nu_{s+t}
   \end{align} 
   which proves the lemma. We prove the inequality
   in~\eqref{eqn:nu-filter-reduced} by induction on the size of the partition of
   $\textbf{s}$. If $\textbf{s}=(s)$, then $s\in (M\setminus E)\cup\mathcal{S}$,
   so $\nu_s=\phi_s$ by Lemma~\ref{lem:nu-leq}. If $\textbf{t}=(t_1,\dots,
   t_\ell)\in \respart(t)$, then $(s, \textbf{t})\in \respart(s+t)$ and
   $\textbf{r} = (\textbf{t}, s)=(t_1,\dots,t_\ell, s)\in \respart(s+t)$. It
   follows that
   \[ 
      [\phi_s, [\phi_{\textbf{t}}]] = [[\phi_{{\bf t}}], \phi_s] = [\phi_{\textbf{r}}] \leq \nu_{s+t}.
   \]
   Therefore, $[\nu_t,\phi_s]\leq \nu_{s+t}$, and since $\phi_s=\nu_s$, 
   \[ 
      [\nu_s,\nu_t]= [\nu_t,\phi_s] \leq \nu_{s+t}.
   \]

   Now we proceed by induction on $k\geq 2$, the size of the partition ${\bf
   s}=(s_1,\dots,s_k)\in \respart(s)$. Let ${\bf s}'=(s_1,\dots,s_{k-1})$, and
   let $A=[\phi_{{\bf s}'}]$, $B=\phi_{s_k}$, and $C=[\phi_{{\bf t}}]$. Then
   \[ 
      [ [\phi_{{\bf s}}],[\phi_{{\bf t}}]] = [A,B,C].
   \]
   Since $({\bf s},{\bf t})\in \respart(s+t)$, all permutations of $({\bf
   s},{\bf t})$ are also contained in $ \respart(s+t)$. Hence,
   $(t_1,\dots,t_\ell,s_k,s_1,\dots,s_{k-1})\in \respart(s+t)$. If ${\bf
   t}'=(t_1,\dots,t_\ell,s_k)$, then by the induction hypothesis
   \[ 
      [B,C,A]=[C,B,A]=[[\phi_{{\bf t}'}], [\phi_{{\bf s}'}]]\leq \nu_{s+t}. 
   \]
   Although $-s_k$ may not be contained $M$, we let $s-s_k$ denote
   $s_1+\cdots+s_{k-1}$. Again, by the induction hypothesis
   \[ 
      [C,A,B] \leq [ \nu_{s-s_k+t}, \phi_{s_k}]\leq \nu_{s+t}.
   \]
   By the Three Subgroups Lemma,
   \[ 
      [[\phi_{{\bf s}}],[\phi_{{\bf t}}]] = [A,B,C] \leq [B,C,A][C,A,B]\leq \nu_{s+t}.
   \]
   Therefore, in this case, $[\nu_s,\nu_t]\leq \nu_{s+t}$.
\end{proof}

We fix the following notation. For $H\in\im(\phi)$, define 
\begin{align*}
   \supp{\phi}{H} &= \{ s\in M \mid \phi_s = H \}, \\ 
   \psupp{\phi}{H} &= \{ s+t \mid s\in \supp{\phi}{H}, t\ne 0\} \cap \supp{\phi}{H}. 
\end{align*}
We generalize this to subsets $\mathcal{H}\subseteq \im(\phi)$, where
$\supp{\phi}{\mathcal{H}}$ is the disjoint union of $\supp{\phi}{H}$, for
$H\in\mathcal{H}$ and similarly for $\psupp{\phi}{\mathcal{H}}$. We use the next
lemma to show that $\widetilde{\nu}$ is progressive.

\begin{lem}\label{lem:minimal-elt}
   If $\phi$ is a progressive $(M, G)$-filter, then for all
   $H\in\im(\phi)\setminus\{\langle 1\rangle\}$,
   $\supp{\phi}{H}\setminus\psupp{\phi}{H}$ is finite and nonempty.
\end{lem}

\begin{proof}
   Suppose $\supp{\phi}{H}=\psupp{\phi}{H}$, and let $s_0\in \supp{\phi}{H}$. As
   $s_0\in \psupp{\phi}{H}$, there exists $s_1\in \supp{\phi}{H}$ and $t_1\ne 0$
   such that $s_0=s_1+t_1$. By induction, there exists $s_{i+1}\in
   \supp{\phi}{H}$ and $t_{i+1}\ne 0$ such that $s_i = s_{i+1} + t_{i+1}$, for
   $i\geq 1$. Since $\phi$ is progressive and $H\ne \langle 1\rangle$, each
   $s_i$ is unit-cancellative. Therefore, $s_i\ne s_{i+1}$. Because $t_i\ne 0$
   for all $i\geq 1$, it follows that $s_0\succ_+ s_1 \succ_+ \cdots$ is an
   infinite descending chain in $M$ (and $\prec_+$ is transitive on the subset
   of unit-cancellative elements). By Corollary~\ref{cor:unit-cancel-dcc-anti},
   this is a contradiction. Therefore, $\supp{\phi}{H}\ne \psupp{\phi}{H}$. By
   definition, $\supp{\phi}{H}\setminus\psupp{\phi}{H}$ is an
   $\preceq_+$-antichain. By Corollary~\ref{cor:unit-cancel-dcc-anti},
   $\supp{\phi}{H}\setminus\psupp{\phi}{H}$ must be finite.
\end{proof}

\begin{prop}\label{prop:hat-filter}
   Suppose $\phi$ is a finite, progressive $(M, G)$-filter. If
   $\mathcal{H}\subseteq \im(\phi)\setminus\{\langle 1\rangle\}$,
   $E=\psupp{\phi}{\mathcal{H}}$, and $\widetilde{\nu}=\widetilde{\nu}(E)$, then
   $\widetilde{\nu}$ is a finite, progressive $(M, G)$-filter such that
   $\im(\phi)\subseteq\im\left(\widetilde{\nu}\right)$.
\end{prop}

\begin{proof}
   By Lemma~\ref{lem:nu-filter}, for all $s,t\in M$,
   \[ 
      \left[ \widetilde{\nu}_s,\widetilde{\nu}_t\right] = \prod_{s\preceq u}\prod_{t\preceq w} [\nu_u,\nu_w] \leq \prod_{s\preceq u}\prod_{t\preceq w} \nu_{u+w} \leq \widetilde{\nu}_{s+t} .
   \]
   If $s\preceq t$, then 
   \[ 
      \widetilde{\nu}_s = \prod_{s\preceq u} \nu_u\geq \prod_{s\preceq t\preceq w} \nu_w = \widetilde{\nu}_t.
   \]
   Therefore, $\widetilde{\nu}$ is an $(M, G)$-filter. By
   Lemma~\ref{lem:hat-order}, $\widetilde{\nu}_s=\phi_s$ for each $s\in
   (M\setminus E)\cup \mathcal{S}$. By Lemma~\ref{lem:minimal-elt}, since
   $\langle 1\rangle\notin \mathcal{H}$ and $\phi$ is progressive, for each
   $H\in\mathcal{H}$ there exists
   $t_H\in\supp{\phi}{H}\setminus\psupp{\phi}{H}$. Since
   $\psupp{\phi}{\mathcal{H}}$ is the disjoint union of sets $\psupp{\phi}{H}$,
   it follows that for all $H\in \mathcal{H}$, $t_H\in\supp{\phi}{\mathcal{H}}
   \setminus \psupp{\phi}{\mathcal{H}}$. Thus, $\phi_{t_H} = H =
   \widetilde{\nu}_{t_H}$, so
   $\im(\phi)\subseteq\im\left(\widetilde{\nu}\right)$. Since $\phi$ is finite,
   so is $\widetilde{\nu}$. Because $\langle 1\rangle\notin\mathcal{H}$, all
   sinks $s\in M$ are contained in $M\setminus E$. From
   Lemma~\ref{lem:hat-order}, since $s$ is a sink,
   $\widetilde{\nu}_s=\phi_s=\langle 1\rangle$, so $\widetilde{\nu}$ is
   progressive. 
\end{proof}

Now we need some lemmas toward inertia. Recall, $\phi_s$ is not an inert
subgroup of $\phi$ if there exists $B\subseteq \mathfrak{B}$ such that
$\partial\phi_s = \langle B\rangle$. The way we show a subgroup $\phi_s$ is not
an inert subgroup of $\phi$ is to prove that $\partial\phi_s$ is generated by
subgroups strictly contained in $\phi_s$. Proving this for all $s\in M$,
together with the finiteness of $\phi$, yields our desired result: $\phi$ is
inertia-free. 

We define some length functions that we only use for the proof of the next
lemma. For $\textbf{r}\in \respart(x)$, let $\ell(\textbf{r})\in\N$ denote the
number of terms in $\textbf{r}$. Define 
\begin{equation*}
   \ell( \respart(x)) = \min_{\textbf{r}\in \respart(x)} \ell(\textbf{r}). 
\end{equation*}

\begin{lem}\label{lem:nu-finite-supp}
   Suppose $\phi$ is a finite, progressive $(M,G)$-filter. Let
   $\mathcal{H}=\im(\phi)\setminus\{\langle 1\rangle\}$,
   $E=\psupp{\phi}{\mathcal{H}}$, and $\nu=\nu(E)$. For all nontrivial
   $H\in\im(\nu)$, $\supp{\nu}{H}$ is finite. 
\end{lem}

\begin{proof}
   By Lemma~\ref{lem:minimal-elt}, $\supp{\phi}{H}\setminus\psupp{\phi}{H}$ is
   finite and possibly empty if $H\notin\im(\phi)$. So if $|\supp{\nu}{H}|$ is
   infinite, then there exists an infinite set $T\subseteq \supp{\nu}{H}$ such
   that $T\subseteq E = \psupp{\phi}{\mathcal{H}}$. Since $\phi$ is progressive,
   $T$ only contains unit-cancellative elements. Therefore, for all $t\in T$ and
   every $\textbf{s}=(s_1, \dots, s_k)\in R_E(t)$, each $s_i$ is
   unit-cancellative by Proposition~\ref{prop:unit-cancel}. Since $\phi$ is
   finite and $M$ is finitely generated, the set $M\setminus (E \cup
   \supp{\phi}{\langle 1\rangle})\cup \mathcal{S}$ is finite by
   Lemma~\ref{lem:minimal-elt} and generates every element of the infinite set
   $T$. Thus, for all $N\geq 1$ there exists only finitely many $t\in T$ such
   that $\ell(R_E(t))\leq N$. Since $G$ is nilpotent, there exists $c\geq 1$
   such that all commutators of weight $c+1$ are trivial. Therefore, there are
   only finitely $t\in T$ such that 
   \begin{align*}
      H &= \prod_{\textbf{s} \in \respart(t)} [\phi_{\textbf{s}}],
   \end{align*}
   which is a contradiction. Hence, no such $T$ exists. 
\end{proof}

\begin{lem}\label{lem:finite-supp}
   Let $\phi$ be a progressive $(M, G)$-filter and
   $H\in\im(\phi)\setminus\{\langle 1\rangle\}$. If $|\supp{\phi}{H}|$ is
   finite, then there exists $s\in \supp{\phi}{H}$ and $I_s\subseteq M$ such
   that for all $t\in I_s$, $\phi_t < H$ and $\partial\phi_s = \langle \phi_t
   \mid t\in I_s\rangle$. 
\end{lem}

\begin{proof}
   Let $s\in \supp{\phi}{H}$ be maximal with respect to $\preceq$. We claim that
   $I_s = \{s+t\mid t\ne 0\}$ suffices. Since $\phi$ is progressive and $H\ne
   \langle 1\rangle$, the element $s$ is not a sink. In particular, $s\notin
   I_s$. By the maximality of $s$, for all $t\in I_s$, it follows that $\phi_t <
   H$, and by definition $\partial\phi_s = \langle \phi_t \mid t\in I_s\rangle$.
\end{proof}

Now we are ready to prove that we can refresh inert subgroups and construct more
vigorous filters. 

\begin{thm}\label{thm:nu-filter}
   If $\phi$ is a finite progressive $(M, G)$-filter, then there exists a finite
   progressive, inertia-free $(M, G)$-filter $\widetilde{\nu}$ such that
   $\im(\phi)\subseteq\im\left(\widetilde{\nu}\right)$. 
\end{thm}

\begin{proof}
   We first show that all nontrivial minimal subgroups are not inert subgroups
   of $\widetilde{\nu}$. Then we show that if $\widetilde{\nu}_s$ has the
   property that every $\widetilde{\nu}_{s+t}<\widetilde{\nu}_s$ is not an inert
   subgroup of $\widetilde{\nu}$, then $\widetilde{\nu}_s$ is not an inert
   subgroup of $\widetilde{\nu}$. Let $\mathcal{H}=\im(\phi)\setminus\{\langle
   1\rangle\}$. Set $E = \psupp{\phi}{\mathcal{H}}$, and define $\nu=\nu(E)$.
   Let $\widetilde{\nu}=\widetilde{\nu}(E)$ denote the $(M, G)$-filter from
   Proposition~\ref{prop:hat-filter}, see also~\eqref{eqn:nu-filter}. By
   Proposition~\ref{prop:hat-filter}, $\widetilde{\nu}$ is finite, progressive,
   and $\im(\phi)\subseteq\im(\widetilde{\nu})$. It remains to prove that
   $\widetilde{\nu}$ is inertia-free, and we do so by induction up the
   $\mathfrak{B}$-sequence. The base case, $\langle 1\rangle$, is not an inert
   subgroup by definition. Suppose $n\geq 0$, and let
   $H\in\im(\widetilde{\nu})\setminus\mathfrak{B}_n$ be minimal. We show that
   $H\in \mathfrak{B}_{n+1}$. This implies that $\widetilde{\nu}$ is
   inertia-free since $\widetilde{\nu}$ is finite. 
   
   By the minimality of $H$, it is enough to show that there exists $s\in
   \supp{\widetilde{\nu}}{H}$ such that $\partial\widetilde{\nu}_s$ is generated
   by subgroups $\widetilde{\nu}_{s+t} < H$. If $|\supp{\widetilde{\nu}}{H}|$ is
   finite, then apply Lemma~\ref{lem:finite-supp}, and if there exists
   $s\in\supp{\widetilde{\nu}}{H}$ such that $\partial\widetilde{\nu}_s\ne
   \widetilde{\nu}_s$, then every $\widetilde{\nu}_{s+t} < H$ for $t\ne 0$. In
   both cases, we conclude $H\in\mathfrak{B}_{n+1}$. Therefore, we assume that
   $|\supp{\widetilde{\nu}}{H}|$ has infinite cardinality and for all
   $s\in\supp{\widetilde{\nu}}{H}$, $\partial\widetilde{\nu}_s=H$. 

   By Lemma~\ref{lem:nu-finite-supp}, every nontrivial $K\in\im(\nu)$ has finite
   support. Since for all $s\in\supp{\widetilde{\nu}}{H}$, 
   \begin{align}\label{eqn:H}
      H &= \prod_{s\preceq t} \nu_t,
   \end{align}
   all but finitely many $t$ from~\eqref{eqn:H} satisfy $\nu_t \ne \langle
   1\rangle$. Thus, there exists a finite (and hence minimal) set $T\subseteq M$
   such that for all $s\in \supp{\widetilde{\nu}}{H}$ and for all $t\in T$,
   either $s\preceq t$ or $s$ and $t$ are incomparable and $H = \prod_{t\in T}
   \nu_t$. By definition, 
   \[ 
      H = \prod_{\substack{t\in T \\ t\preceq u}} \nu_u 
      = \prod_{t\in T} \widetilde{\nu}_t. 
   \] 
   If for all $t\in T$, $\widetilde{\nu}_t <H$, then we are done. On the other
   hand, if there exists $t\in T$ such that $\widetilde{\nu}_t=H$, then $t$ is a
   maximal element of $\supp{\widetilde{\nu}}{H}$. In particular, for all $u\in
   M$ with $u\ne 0$, $\widetilde{\nu}_{t+u} < H$. Thus,
   $H\in\mathfrak{B}_{n+1}\subseteq \mathfrak{B}$. 
\end{proof}

\begin{remark}
   It is unknown what conditions are sufficient for $\phi$ to guarantee that
   $\widetilde{\nu}$ from Theorem~\ref{thm:nu-filter} is faithful.
   Example~\ref{ex:not-faithful} shows that $\widetilde{\nu}$ need not be
   faithful. In the context of algorithms for isomorphism testing, an algorithm
   that refines filters and always returns a faithful filter would be
   sufficient. However, it may not be possible for such an algorithm to always
   refine; it may have to remove some subgroups from the image.  
\end{remark}

\subsection{Some examples}\label{sec:some-examples}

In Section~\ref{sec:examples}, there are three examples of filters with inert
subgroups, two of which are progressive. We illustrate the construction of
$\widetilde{\nu}$ with these progressive filters.

\begin{ex}
   First we consider Example~\ref{ex:trivial}. Suppose $G=H(\mathbb{Z})$ and
   $M=(\N^2,\preceq_\ell)$, ordered by the lex ordering. A generating set for
   $M$ is $\mathcal{S}=\{(1,0), (0, 1)\}$. The following function is an $(M,
   G)$-filter given by 
   \[ 
      \phi_s = \left\{ \begin{array}{ll} 
         G & s\prec_\ell (2,0), \\ 
         G' & (2,0)\preceq_\ell s \prec_\ell (3,0), \\ 
         1 & (3,0) \preceq_\ell s, 
      \end{array}\right. 
   \]
   and every subgroup in $\im(\phi)$ has infinite support. Set $\mathcal{H}=\{G,
   G'\}$, so 
   \[ 
      \psupp{\phi}{\mathcal{H}} = \{s \in M \mid (0, 0)\prec_\ell s\prec_\ell (2,0) \} \cup \{ s\in M \mid (2,0)\prec_\ell s\prec_\ell (3,0) \} . 
   \]
   With $E=\psupp{\phi}{\mathcal{H}}$, we plot both $\nu=\nu(E)$ and
   $\widetilde{\nu}=\widetilde{\nu}(E)$. The $(\N^2,G)$-filter we construct is
   the same as the filter we would construct with the tools
   from~\cite{M:efficient-filters} since $M$ is totally ordered by $\preceq$.
   
   \begin{figure}[ht]
      \centering
      \begin{subfigure}[b]{0.3\textwidth}
         \centering
         \begin{tikzpicture}
            \pgfmathsetmacro{\myscale}{0.5}
            \node (ptx) at (-0.328,-0.2) {};
            \node (pty) at (-0.2,-0.328) {};
            \node (x) at (0.4 + 3*\myscale,-0.2) {};
            \node (y) at (-0.2, 0.45 + 3*\myscale) {};
            \draw[->] (ptx) -- (x);
            \draw[->] (pty) -- (y);
         
            \node (x0) at (0,-0.5) {0};
            \node (x1) at (\myscale,-0.5) {1};
            \node (x2) at (2*\myscale,-0.5) {2};
            \node (x3) at (3*\myscale,-0.5) {3};
            \node (y0) at (-0.5,0) {0};
            \node (y1) at (-0.5,\myscale) {1};
            \node (y2) at (-0.5,2*\myscale) {2};
            \node (y3) at (-0.5,3*\myscale) {3};
            
            \node (G) at (0,0) {$G$};
            \node (H) at (\myscale,0) {$G$};
            \node (G3) at (2*\myscale,0) {$G'$};
            \node (G4) at (3*\myscale,0) {$1$};
            \node (K) at (0,\myscale) {$G$};
            \node (G2) at (\myscale,\myscale) {$G$};
            \node (G32) at (2*\myscale,\myscale) {$G'$};
            \node (G42) at (3*\myscale,\myscale) {$1$};
            \node (L) at (0,2*\myscale) {$G$};
            \node (G33) at (\myscale,2*\myscale) {$G$};
            \node (G43) at (2*\myscale,2*\myscale) {$G'$};
            \node (G44) at (3*\myscale,2*\myscale) {$1$};
            \node (L) at (0,3*\myscale) {$G$};
            \node (G33) at (\myscale,3*\myscale) {$G$};
            \node (G43) at (2*\myscale,3*\myscale) {$G'$};
            \node (G44) at (3*\myscale,3*\myscale) {$1$};
         \end{tikzpicture}
         \caption{$\phi$}
         \label{fig:trivial-phi}
      \end{subfigure}\quad%
      \begin{subfigure}[b]{0.3\textwidth}
         \centering
         \begin{tikzpicture}
            \pgfmathsetmacro{\myscale}{0.5}
            \node (ptx) at (-0.328,-0.2) {};
            \node (pty) at (-0.2,-0.328) {};
            \node (x) at (0.4 + 3*\myscale,-0.2) {};
            \node (y) at (-0.2, 0.45 + 3*\myscale) {};
            \draw[->] (ptx) -- (x);
            \draw[->] (pty) -- (y);
         
            \node (x0) at (0,-0.5) {0};
            \node (x1) at (\myscale,-0.5) {1};
            \node (x2) at (2*\myscale,-0.5) {2};
            \node (x3) at (3*\myscale,-0.5) {3};
            \node (y0) at (-0.5,0) {0};
            \node (y1) at (-0.5,\myscale) {1};
            \node (y2) at (-0.5,2*\myscale) {2};
            \node (y3) at (-0.5,3*\myscale) {3};
            
            \node[inner sep=2pt, circle, draw=black] (G) at (0,0) {$G$};
            \node[inner sep=2pt, circle, draw=black] (H) at (\myscale,0) {$G$};
            \node[inner sep=1pt, circle, draw=black] (G3) at (2*\myscale,0) {$G'$};
            \node (G4) at (3*\myscale,0) {$1$};
            \node[inner sep=2pt, circle, draw=black] (K) at (0,\myscale) {$G$};
            \node (G2) at (\myscale,\myscale) {$G'$};
            \node (G32) at (2*\myscale,\myscale) {$1$};
            \node (G42) at (3*\myscale,\myscale) {$1$};
            \node (L) at (0,2*\myscale) {$G'$};
            \node (G33) at (\myscale,2*\myscale) {$1$};
            \node (G43) at (2*\myscale,2*\myscale) {$1$};
            \node (G44) at (3*\myscale,2*\myscale) {$1$};
            \node (L) at (0,3*\myscale) {$1$};
            \node (G33) at (\myscale,3*\myscale) {$1$};
            \node (G43) at (2*\myscale,3*\myscale) {$1$};
            \node (G44) at (3*\myscale,3*\myscale) {$1$};
         \end{tikzpicture}
         \caption{$\nu$}
         \label{fig:trivial-nu}
      \end{subfigure}\quad%
      \begin{subfigure}[b]{0.3\textwidth}
         \centering
         \begin{tikzpicture}
            \pgfmathsetmacro{\myscale}{0.5}
            \node (ptx) at (-0.328,-0.2) {};
            \node (pty) at (-0.2,-0.328) {};
            \node (x) at (0.4 + 3*\myscale,-0.2) {};
            \node (y) at (-0.2, 0.45 + 3*\myscale) {};
            \draw[->] (ptx) -- (x);
            \draw[->] (pty) -- (y);
         
            \node (x0) at (0,-0.5) {0};
            \node (x1) at (\myscale,-0.5) {1};
            \node (x2) at (2*\myscale,-0.5) {2};
            \node (x3) at (3*\myscale,-0.5) {3};
            \node (y0) at (-0.5,0) {0};
            \node (y1) at (-0.5,\myscale) {1};
            \node (y2) at (-0.5,2*\myscale) {2};
            \node (y3) at (-0.5,3*\myscale) {3};
            
            \node (G) at (0,0) {$G$};
            \node (H) at (\myscale,0) {$G$};
            \node (G3) at (2*\myscale,0) {$G'$};
            \node (G4) at (3*\myscale,0) {$1$};
            \node (K) at (0,\myscale) {$G$};
            \node (G2) at (\myscale,\myscale) {$G'$};
            \node (G32) at (2*\myscale,\myscale) {$1$};
            \node (G42) at (3*\myscale,\myscale) {$1$};
            \node (L) at (0,2*\myscale) {$G$};
            \node (G33) at (\myscale,2*\myscale) {$G'$};
            \node (G43) at (2*\myscale,2*\myscale) {$1$};
            \node (G44) at (3*\myscale,2*\myscale) {$1$};
            \node (L) at (0,3*\myscale) {$G$};
            \node (G33) at (\myscale,3*\myscale) {$G'$};
            \node (G43) at (2*\myscale,3*\myscale) {$1$};
            \node (G44) at (3*\myscale,3*\myscale) {$1$};
         \end{tikzpicture}
         \caption{$\widetilde{\nu}$}\label{fig:both-nu}
         \label{fig:trivial-tilde-nu}
      \end{subfigure}
      \caption{We refresh the $(M, G)$-filter $\phi$, defined in Example~\ref{ex:trivial}. First, we construct the function $\nu$ in Figure~\ref{fig:trivial-nu}. We circle the elements in $(M\setminus E)\cup\mathcal{S}$ that do not evaluate to the trivial group. Then we construct the $(M,G)$-filter $\widetilde{\nu}$, as seen in Figure~\ref{fig:trivial-tilde-nu}.}
      \label{fig:refresh-trivial}
   \end{figure}
\end{ex}

\begin{ex}\label{ex:not-faithful}
   Now we consider Example~\ref{ex:just-bijection}. Recall, $G=H(K)$, for some finite field $K$, $M=(\N^2, \preceq_+)$, and for all $s\in M$, 
   \begin{align*}
      \phi_s &= \left\{ \begin{array}{ll}
         G & i = 0, \\
         G' & i = 1\text{ or } i+j\leq 4, \\
         1 & \text{otherwise}.
      \end{array}\right.
   \end{align*}
   Again, we set $\mathcal{H}=\{G, G'\}$, so 
   \[ \psupp{\phi}{\mathcal{H}} = \{(0, j) \in M\mid j\geq 1 \} \cup \{ (i, j)\in M \mid 2\leq i+j \leq 4\; \text{ or }\; j\geq i = 1 \} . \]
   All the functions are plotted in Figure~\ref{fig:refresh-bijection}. In particular, $\widetilde{\nu}$ is not a faithful filter as 
   \begin{align*}
      \widetilde{\nu}_{(1,0)} \setminus \partial\widetilde{\nu}_{(1,0)} = \widetilde{\nu}_{(0,2)} \setminus \partial\widetilde{\nu}_{(0,2)} = G' \setminus \{1\}.
   \end{align*}
   Indeed, $L(\widetilde{\nu}) \cong G/G'\oplus G' \oplus G'$, which is not in bijection with $G$. 

   \begin{figure}[ht]
      \centering
      \begin{subfigure}[b]{0.3\textwidth}
         \centering
            \begin{tikzpicture}
               \pgfmathsetmacro{\myscale}{0.5}
               \node (ptx) at (-0.328,-0.2) {};
               \node (pty) at (-0.2,-0.328) {};
               \node (x) at (0.4 + 5*\myscale,-0.2) {};
               \node (y) at (-0.2, 0.45 + 4*\myscale) {};
               \draw[->] (ptx) -- (x);
               \draw[->] (pty) -- (y);
            
               \node (x0) at (0,-0.5) {0};
               \node (x1) at (\myscale,-0.5) {1};
               \node (x2) at (2*\myscale,-0.5) {2};
               \node (x3) at (3*\myscale,-0.5) {3};
               \node (x4) at (4*\myscale,-0.5) {4};
               \node (x5) at (5*\myscale,-0.5) {5};
               \node (y0) at (-0.5,0) {0};
               \node (y1) at (-0.5,\myscale) {1};
               \node (y2) at (-0.5,2*\myscale) {2};
               \node (y3) at (-0.5,3*\myscale) {3};
               \node (y4) at (-0.5,4*\myscale) {4};
               
               \node (G) at (0,0) {$G$};
               \node (H) at (\myscale,0) {$G'$};
               \node (G3) at (2*\myscale,0) {$G'$};
               \node (G4) at (3*\myscale,0) {$G'$};
               \node (G40) at (4*\myscale,0) {$G'$};
               \node (G40) at (5*\myscale,0) {$1$};
               \node (K) at (0,\myscale) {$G$};
               \node (G2) at (\myscale,\myscale) {$G'$};
               \node (G32) at (2*\myscale,\myscale) {$G'$};
               \node (G42) at (3*\myscale,\myscale) {$G'$};
               \node (G41) at (4*\myscale,\myscale) {$1$};
               \node (G41) at (5*\myscale,\myscale) {$1$};
               \node (L) at (0,2*\myscale) {$G$};
               \node (G33) at (\myscale,2*\myscale) {$G'$};
               \node (G43) at (2*\myscale,2*\myscale) {$G'$};
               \node (G44) at (3*\myscale,2*\myscale) {$1$};
               \node (G43) at (4*\myscale,2*\myscale) {$1$};
               \node (G43) at (5*\myscale,2*\myscale) {$1$};
               \node (14) at (0,3*\myscale) {$G$};
               \node (15) at (\myscale,3*\myscale) {$G'$};
               \node (16) at (2*\myscale,3*\myscale) {$1$};
               \node (17) at (3*\myscale,3*\myscale) {$1$};
               \node (18) at (4*\myscale,3*\myscale) {$1$};
               \node (18) at (5*\myscale,3*\myscale) {$1$};
               \node (14) at (0,4*\myscale) {$G$};
               \node (15) at (\myscale,4*\myscale) {$G'$};
               \node (16) at (2*\myscale,4*\myscale) {$1$};
               \node (17) at (3*\myscale,4*\myscale) {$1$};
               \node (18) at (4*\myscale,4*\myscale) {$1$};
               \node (18) at (5*\myscale,4*\myscale) {$1$};
            \end{tikzpicture}
            \caption{$\phi$}
            \label{fig:bijection-phi}
      \end{subfigure}\quad%
      \begin{subfigure}[b]{0.3\textwidth}
         \centering
         \begin{tikzpicture}
            \pgfmathsetmacro{\myscale}{0.5}
            \node (ptx) at (-0.328,-0.2) {};
            \node (pty) at (-0.2,-0.328) {};
            \node (x) at (0.4 + 5*\myscale,-0.2) {};
            \node (y) at (-0.2, 0.45 + 4*\myscale) {};
            \draw[->] (ptx) -- (x);
            \draw[->] (pty) -- (y);
         
            \node (x0) at (0,-0.5) {0};
            \node (x1) at (\myscale,-0.5) {1};
            \node (x2) at (2*\myscale,-0.5) {2};
            \node (x3) at (3*\myscale,-0.5) {3};
            \node (x4) at (4*\myscale,-0.5) {4};
            \node (x5) at (5*\myscale,-0.5) {5};
            \node (y0) at (-0.5,0) {0};
            \node (y1) at (-0.5,\myscale) {1};
            \node (y2) at (-0.5,2*\myscale) {2};
            \node (y3) at (-0.5,3*\myscale) {3};
            \node (y4) at (-0.5,4*\myscale) {4};
            
            \node[inner sep=2pt, circle, draw=black] (G) at (0,0) {$G$};
            \node[inner sep=1pt, circle, draw=black] (H) at (\myscale,0) {$G'$};
            \node (G3) at (2*\myscale,0) {$1$};
            \node (G4) at (3*\myscale,0) {$1$};
            \node (G40) at (4*\myscale,0) {$1$};
            \node (G40) at (5*\myscale,0) {$1$};
            \node[inner sep=2pt, circle, draw=black] (K) at (0,\myscale) {$G$};
            \node (G2) at (\myscale,\myscale) {$1$};
            \node (G32) at (2*\myscale,\myscale) {$1$};
            \node (G42) at (3*\myscale,\myscale) {$1$};
            \node (G41) at (4*\myscale,\myscale) {$1$};
            \node (G41) at (5*\myscale,\myscale) {$1$};
            \node (L) at (0,2*\myscale) {$G'$};
            \node (G33) at (\myscale,2*\myscale) {$1$};
            \node (G43) at (2*\myscale,2*\myscale) {$1$};
            \node (G44) at (3*\myscale,2*\myscale) {$1$};
            \node (G43) at (4*\myscale,2*\myscale) {$1$};
            \node (G43) at (5*\myscale,2*\myscale) {$1$};
            \node (14) at (0,3*\myscale) {$1$};
            \node (15) at (\myscale,3*\myscale) {$1$};
            \node (16) at (2*\myscale,3*\myscale) {$1$};
            \node (17) at (3*\myscale,3*\myscale) {$1$};
            \node (18) at (4*\myscale,3*\myscale) {$1$};
            \node (18) at (5*\myscale,3*\myscale) {$1$};
            \node (14) at (0,4*\myscale) {$1$};
            \node (15) at (\myscale,4*\myscale) {$1$};
            \node (16) at (2*\myscale,4*\myscale) {$1$};
            \node (17) at (3*\myscale,4*\myscale) {$1$};
            \node (18) at (4*\myscale,4*\myscale) {$1$};
            \node (18) at (5*\myscale,4*\myscale) {$1$};
         \end{tikzpicture}
         \caption{$\nu$}
         \label{fig:bijection-nu}
      \end{subfigure}\quad%
      \begin{subfigure}[b]{0.3\textwidth}
         \centering
         \begin{tikzpicture}
            \pgfmathsetmacro{\myscale}{0.5}
            \node (ptx) at (-0.328,-0.2) {};
            \node (pty) at (-0.2,-0.328) {};
            \node (x) at (0.4 + 5*\myscale,-0.2) {};
            \node (y) at (-0.2, 0.45 + 4*\myscale) {};
            \draw[->] (ptx) -- (x);
            \draw[->] (pty) -- (y);
         
            \node (x0) at (0,-0.5) {0};
            \node (x1) at (\myscale,-0.5) {1};
            \node (x2) at (2*\myscale,-0.5) {2};
            \node (x3) at (3*\myscale,-0.5) {3};
            \node (x4) at (4*\myscale,-0.5) {4};
            \node (x5) at (5*\myscale,-0.5) {5};
            \node (y0) at (-0.5,0) {0};
            \node (y1) at (-0.5,\myscale) {1};
            \node (y2) at (-0.5,2*\myscale) {2};
            \node (y3) at (-0.5,3*\myscale) {3};
            \node (y4) at (-0.5,4*\myscale) {4};
            
            \node (G) at (0,0) {$G$};
            \node (H) at (\myscale,0) {$G'$};
            \node (G3) at (2*\myscale,0) {$1$};
            \node (G4) at (3*\myscale,0) {$1$};
            \node (G40) at (4*\myscale,0) {$1$};
            \node (G40) at (5*\myscale,0) {$1$};
            \node (K) at (0,\myscale) {$G$};
            \node (G2) at (\myscale,\myscale) {$1$};
            \node (G32) at (2*\myscale,\myscale) {$1$};
            \node (G42) at (3*\myscale,\myscale) {$1$};
            \node (G41) at (4*\myscale,\myscale) {$1$};
            \node (G41) at (5*\myscale,\myscale) {$1$};
            \node (L) at (0,2*\myscale) {$G'$};
            \node (G33) at (\myscale,2*\myscale) {$1$};
            \node (G43) at (2*\myscale,2*\myscale) {$1$};
            \node (G44) at (3*\myscale,2*\myscale) {$1$};
            \node (G43) at (4*\myscale,2*\myscale) {$1$};
            \node (G43) at (5*\myscale,2*\myscale) {$1$};
            \node (14) at (0,3*\myscale) {$1$};
            \node (15) at (\myscale,3*\myscale) {$1$};
            \node (16) at (2*\myscale,3*\myscale) {$1$};
            \node (17) at (3*\myscale,3*\myscale) {$1$};
            \node (18) at (4*\myscale,3*\myscale) {$1$};
            \node (18) at (5*\myscale,3*\myscale) {$1$};
            \node (14) at (0,4*\myscale) {$1$};
            \node (15) at (\myscale,4*\myscale) {$1$};
            \node (16) at (2*\myscale,4*\myscale) {$1$};
            \node (17) at (3*\myscale,4*\myscale) {$1$};
            \node (18) at (4*\myscale,4*\myscale) {$1$};
            \node (18) at (5*\myscale,4*\myscale) {$1$};
         \end{tikzpicture}
         \caption{$\widetilde{\nu}$}
         \label{fig:bijection-tilde-nu}
      \end{subfigure}
      \caption{We refresh the $(M, G)$-filter $\phi$, defined in Example~\ref{ex:just-bijection}. First, we construct the function $\nu$ in Figure~\ref{fig:bijection-nu}. We circle the elements in $(M\setminus E)\cup\mathcal{S}$ that do not evaluate to the trivial group. Then we construct the $(M,G)$-filter $\widetilde{\nu}$, as seen in Figure~\ref{fig:bijection-tilde-nu}.}
      \label{fig:refresh-bijection}
   \end{figure}
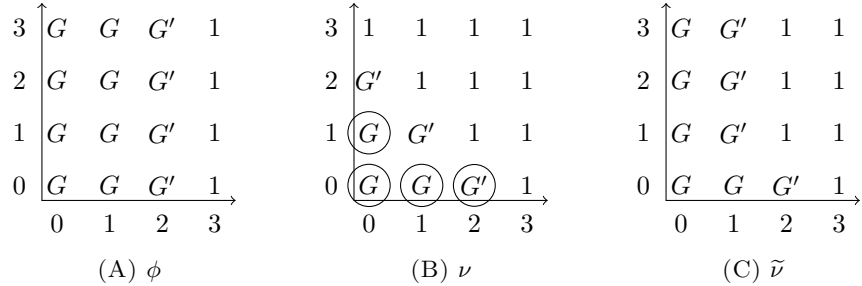
\end{ex}

\subsection{Proof of Theorem~\ref{thm:general-progressive}}

In Sections~\ref{sec:refreshing-filters} and~\ref{sec:some-examples}, we
required the monoid structure of $M$ to be nice enough so our construction would
remove inertness. If we change the monoid the filter is defined over, then we
require fewer assumptions on $M$, only that $\preceq$ be a partial order instead
of just a pre-order. We move to the free commutative monoid $\N^d$, which
eliminates sinks. Care is needed when constructing a partial order that is
compatible with the partial order on $M$, and that is the objective of the next
lemma.

\begin{lem}\label{lem:fgcp-morphism}
   If $M$ is a conically partially-ordered monoid with $\preceq$ a partial
   order, then there exists an integer $d\in\N$, a partial order $\preceq'$ for
   $\N^d$, and a surjection $\pi:\N^d\rightarrow M$ such that 
   \begin{enumerate}
      \item[$(i)$] $(\N^d, \preceq')$ is a conically partially-ordered monoid
      and 
      \item[$(ii)$] if $s\preceq' t$, then $\pi(s)\preceq \pi(t)$. 
   \end{enumerate}
\end{lem}

\begin{proof}
   Since $M$ is finitely generated, there exists $d\in\mathbb{Z}$ and a
   congruence $\sim$ of $\N^d$ such that $\N^d/\!\!\sim\;\cong M$. Let
   $\pi:\N^d\rightarrow M$ be the induced surjection. Define a partial order
   $\preceq'$ on $\N^d$ as follows. For $s,t\in \N^d$, 
   \[ 
      s\preceq' t \iff (\pi(s)\prec \pi(t)) \text{ or } (\pi(s)=\pi(t) \text{ and } s\preceq_+ t). 
   \]
   Because $\preceq$ is a partial order on $M$, $\prec$ is transitive. Since
   $\pi$ is a monoid homomorphism, $\preceq$ a partial order on $M$, and
   $\preceq_+$ a partial order of $\N^d$, it follows that $\preceq'$ is a
   partial order for $\N^d$. As $0$ is $\preceq'$-minimal in $\N^d$, $(\N^d,
   \preceq')$ is a conically partially-ordered monoid. 
\end{proof}

\begin{proof}[Proof of Theorem~\ref{thm:general-progressive}]
   If $\phi$ is a progressive $(M, G)$-filter, then apply Theorem~\ref{thm:nu-filter}. Otherwise, $\preceq$ is a partial order. By Lemma~\ref{lem:fgcp-morphism}, there exists a conically pre-ordered monoid $(\N^d, \preceq')$ and a surjection $\pi:\N^d\rightarrow M$. 
   Define a function $\rho$ from $\N^d$ into $\Nor(G)$ such that $\rho_s= \phi_{\pi(s)}$.
   For $s,t\in\N^d$, 
   \begin{align*}
      [\rho_s, \rho_t] = [\phi_{\pi(s)}, \phi_{\pi(t)}] \leq \phi_{\pi(s+t)} = \rho_{s+t}. 
   \end{align*}
   In addition, if $s\preceq' t$, then $\pi(s)\preceq \pi(t)$. Thus, $\rho_s = \phi_{\pi(s)} \geq \phi_{\pi(t)} = \rho_t$. 
   Therefore, $\rho$ is an $(\N^d, G)$-filter. Since $\pi$ is surjective, $\im(\rho)=\im(\phi)$. Since every element of $\N^d$ is unit-cancellative, $\rho$ is progressive. Apply Theorem~\ref{thm:nu-filter} to $\rho$.
\end{proof}

\begin{ex}
   We want to fix the last filter in Section~\ref{sec:examples} that has inert subgroups: Example~\ref{ex:not-progressive}. Recall, $G=H(\mathbb{Z})$, $M=(C_{3,1}\times C_{1,1}, \preceq_+)$, and for all $s\in M$, 
   \begin{align*}
      \phi_s &=\left\{ \begin{array}{ll}
         G & i = j = 0,\\
         \gamma_i(G) & i > j = 0,\\
         K & i\leq 2, j=1, \\
         1 & i=3, j=1.
      \end{array}\right.
   \end{align*}
   Because $(0, 1)$ is not unit-cancellative and $\phi_{(0, 1)}=K\ne \langle 1\rangle$, $\phi$ is not progressive. Since $\preceq_+$ is a partial order on $M$, we can fix this by applying Theorem~\ref{thm:general-progressive}. 

   A generating set for $M$ is $\mathcal{S}=\{(1, 0), (0, 1)\}$, so by Lemma~\ref{lem:fgcp-morphism}, there exists a surjection $\pi:\N^2\rightarrow M$ and a partial order $\preceq'$ on $\N^2$. Using this, we define an $(\N^2, G)$-filter $\rho$ such that, for $s=(i, j)\in M$, 
   \begin{align*}
      \rho_s &=\left\{ \begin{array}{ll}
         G & i = j = 0,\\
         \gamma_i(G) & i > j = 0,\\
         K & i\leq 2, j\geq 1, \\
         1 & i\geq 3, j\geq 1.
      \end{array}\right.
   \end{align*}
   Now $\rho$ is progressive, and we can construct an inertia-free $(\N^2, G)$-filter $\widetilde{\nu}$. We plot these four functions in Figure~\ref{fig:refresh-progressive}.

   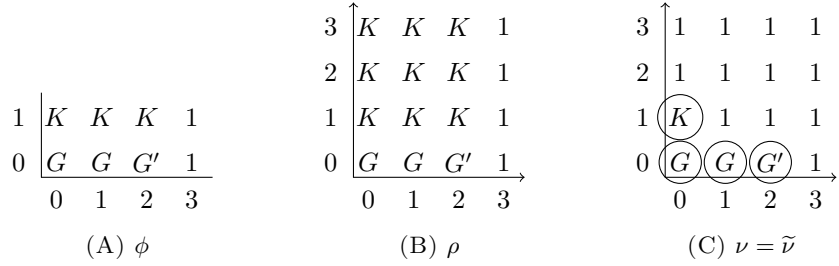
\begin{figure}[ht]
      \centering
      \begin{subfigure}[b]{0.3\textwidth}
         \centering
         \begin{tikzpicture}
            \pgfmathsetmacro{\myscale}{0.5}
            \node (ptx) at (-0.328,-0.2) {};
            \node (pty) at (-0.2,-0.328) {};
            \node (x) at (0.4 + 3*\myscale,-0.2) {};
            \node (y) at (-0.2, 0.45 + 1*\myscale) {};
            \draw[-] (ptx) -- (x);
            \draw[-] (pty) -- (y);
         
            \node (x0) at (0,-0.5) {0};
            \node (x1) at (\myscale,-0.5) {1};
            \node (x2) at (2*\myscale,-0.5) {2};
            \node (x3) at (3*\myscale,-0.5) {3};
            \node (y0) at (-0.5,0) {0};
            \node (y1) at (-0.5,\myscale) {1};
            
            \node (G) at (0,0) {$G$};
            \node (H) at (\myscale,0) {$G$};
            \node (G3) at (2*\myscale,0) {$G'$};
            \node (G4) at (3*\myscale,0) {$1$};
            \node (K) at (0,\myscale) {$K$};
            \node (G2) at (\myscale,\myscale) {$K$};
            \node (G32) at (2*\myscale,\myscale) {$K$};
            \node (G42) at (3*\myscale,\myscale) {$1$};
         \end{tikzpicture}
         \caption{$\phi$}
         \label{fig:progressive-phi}
      \end{subfigure}\quad%
      \begin{subfigure}[b]{0.3\textwidth}
         \centering
         \begin{tikzpicture}
            \pgfmathsetmacro{\myscale}{0.5}
            \node (ptx) at (-0.328,-0.2) {};
            \node (pty) at (-0.2,-0.328) {};
            \node (x) at (0.4 + 3*\myscale,-0.2) {};
            \node (y) at (-0.2, 0.45 + 3*\myscale) {};
            \draw[->] (ptx) -- (x);
            \draw[->] (pty) -- (y);
         
            \node (x0) at (0,-0.5) {0};
            \node (x1) at (\myscale,-0.5) {1};
            \node (x2) at (2*\myscale,-0.5) {2};
            \node (x3) at (3*\myscale,-0.5) {3};
            \node (y0) at (-0.5,0) {0};
            \node (y1) at (-0.5,\myscale) {1};
            \node (y2) at (-0.5,2*\myscale) {2};
            \node (y3) at (-0.5,3*\myscale) {3};
            
            \node (G) at (0,0) {$G$};
            \node (H) at (\myscale,0) {$G$};
            \node (G3) at (2*\myscale,0) {$G'$};
            \node (G4) at (3*\myscale,0) {$1$};
            \node (K) at (0,\myscale) {$K$};
            \node (G2) at (\myscale,\myscale) {$K$};
            \node (G32) at (2*\myscale,\myscale) {$K$};
            \node (G42) at (3*\myscale,\myscale) {$1$};
            \node (K) at (0,2*\myscale) {$K$};
            \node (G2) at (\myscale,2*\myscale) {$K$};
            \node (G32) at (2*\myscale,2*\myscale) {$K$};
            \node (G42) at (3*\myscale,2*\myscale) {$1$};
            \node (K) at (0,3*\myscale) {$K$};
            \node (G2) at (\myscale,3*\myscale) {$K$};
            \node (G32) at (2*\myscale,3*\myscale) {$K$};
            \node (G42) at (3*\myscale,3*\myscale) {$1$};
         \end{tikzpicture}
         \caption{$\rho$}
         \label{fig:progressive-rho}
      \end{subfigure}\quad%
      \begin{subfigure}[b]{0.3\textwidth}
         \centering
         \begin{tikzpicture}
            \pgfmathsetmacro{\myscale}{0.5}
            \node (ptx) at (-0.328,-0.2) {};
            \node (pty) at (-0.2,-0.328) {};
            \node (x) at (0.4 + 3*\myscale,-0.2) {};
            \node (y) at (-0.2, 0.45 + 3*\myscale) {};
            \draw[->] (ptx) -- (x);
            \draw[->] (pty) -- (y);
         
            \node (x0) at (0,-0.5) {0};
            \node (x1) at (\myscale,-0.5) {1};
            \node (x2) at (2*\myscale,-0.5) {2};
            \node (x3) at (3*\myscale,-0.5) {3};
            \node (y0) at (-0.5,0) {0};
            \node (y1) at (-0.5,\myscale) {1};
            \node (y2) at (-0.5,2*\myscale) {2};
            \node (y3) at (-0.5,3*\myscale) {3};
            
            \node[inner sep=2pt, circle, draw=black] (G) at (0,0) {$G$};
            \node[inner sep=2pt, circle, draw=black] (H) at (\myscale,0) {$G$};
            \node[inner sep=1pt, circle, draw=black] (G3) at (2*\myscale,0) {$G'$};
            \node (G4) at (3*\myscale,0) {$1$};
            \node[inner sep=2pt, circle, draw=black] (K) at (0,\myscale) {$K$};
            \node (G2) at (\myscale,\myscale) {$1$};
            \node (G32) at (2*\myscale,\myscale) {$1$};
            \node (G42) at (3*\myscale,\myscale) {$1$};
            \node (K) at (0,2*\myscale) {$1$};
            \node (G2) at (\myscale,2*\myscale) {$1$};
            \node (G32) at (2*\myscale,2*\myscale) {$1$};
            \node (G42) at (3*\myscale,2*\myscale) {$1$};
            \node (K) at (0,3*\myscale) {$1$};
            \node (G2) at (\myscale,3*\myscale) {$1$};
            \node (G32) at (2*\myscale,3*\myscale) {$1$};
            \node (G42) at (3*\myscale,3*\myscale) {$1$};
         \end{tikzpicture}
         \caption{$\nu=\widetilde{\nu}$}
         \label{fig:progressive-nu}
      \end{subfigure}
      \caption{We refresh the $(M, G)$-filter $\phi$, defined in Example~\ref{ex:not-progressive}. First, we construct the progressive $(\N^2, G)$-filter $\rho$ in Figure~\ref{fig:progressive-rho}. Then we construct the function $\nu$ from $\rho$ in Figure~\ref{fig:progressive-nu}. We circle the elements in $(M\setminus E)\cup\mathcal{S}$ that do not evaluate to the trivial group. It turns out that, in this example, $\widetilde{\nu}=\nu$, so we do not plot $\widetilde{\nu}$ separately.}
      \label{fig:refresh-progressive}
   \end{figure}
\end{ex}

\subsection{Proof of Theorem~\ref{thm:Main}}

Theorems~\ref{thm:general-progressive} and~\ref{thm:surjection} imply
Theorem~\ref{thm:Main}, so we finish off the proof of
Theorem~\ref{thm:surjection} now. Assume that $\phi$ is a finite, inertia-free
$(M, G)$-filter.

\begin{definition}
A sequence $\mathcal{Y}$ of elements of an abelian group $\bigoplus_{s\in M}
L_s$ is a \emph{graded pcgs} if 
\begin{enumerate}
   \item[$(i)$] for all $y\in \mathcal{Y}$, there exists $s\in M$ such that $y\in L_s$, and
   \item[$(ii)$] for all $s\in M$, the subsequence $\mathcal{Y}\cap L_s$ is a pcgs for $L_s$. 
\end{enumerate}
\end{definition}

Throughout the remainder of this section we denote the graded pcgs for the abelian group $L(\phi)$ by $\mathcal{Y}$. 
Because $\phi$ is inertia-free, we use the $\mathfrak{B}$-sequence as a method of Noetherian induction to prove the next lemma. 

\begin{lem}\label{lem:weakly-pcgs}
   Suppose $\phi$ is a finite, inertia-free $(M,G)$-filter, and $\mathcal{Y}$ is
   a graded pcgs for $L(\phi)$. If $\mathcal{X}$ is a pre-image of $\mathcal{Y}$
   in $G$, then for all $s\in M$, $\langle\phi_s\cap\mathcal{X}\rangle=\phi_s$
   and $\mathcal{X}$ contains a pcgs for $G$. 
\end{lem} 

\begin{proof}
   Since $\phi$ is finite, $\langle 1\rangle \in\im(\phi)$, so
   $\mathfrak{B}_0=\{\langle 1\rangle \}$. For all $\phi_s\in\mathfrak{B}_0$, it
   follows that $\langle \phi_s\cap\mathcal{X}\rangle = \phi_s$ and
   $\mathcal{X}$ contains a pcgs for $\phi_s$. We assume this holds for all
   subgroups in $\mathfrak{B}_n$, for $n\geq 0$.

   Suppose $\phi_s\in \mathfrak{B}_{n+1}$ for $n\geq 0$. By definition, there
   exists $B\subseteq \mathfrak{B}_n$ such that $\partial\phi_s = \langle
   B\rangle$. By induction, for all $\phi_t\in B$, $\langle \phi_t\cap
   \mathcal{X}\rangle = \phi_t$ and $\mathcal{X}$ contains a subsequence that is
   a pcgs for $\phi_t$. Since $\mathcal{Y}$ is a graded pcgs for $L(\phi)$, the
   subsequence $\mathcal{Y}_s= L_s(\phi)\cap \mathcal{Y}$ is a pcgs for
   $L_s(\phi)$. Let $\mathcal{X}_s$ be the subsequence of $\mathcal{X}$
   corresponding to $\mathcal{Y}_s$ in $G$. Therefore by induction, 
   \[ 
      \langle \phi_s\cap \mathcal{X}\rangle 
      = \langle \mathcal{X}_s \cup (\partial\phi_s\cap \mathcal{X})\rangle 
      = \phi_s. 
   \]
   Since $\phi$ is finite and inertia-free, for every $s\in M$, there exists
   $n\geq 0$ such that $\phi_s\in\mathfrak{B}_n$. Since $G=\langle \phi_s\mid
   s\ne 0\rangle$, then $\mathcal{X}$ contains a pcgs for $G$. 
\end{proof}

\begin{proof}[Proof of Theorem~\ref{thm:surjection}]
   Let $\mathcal{Y}=(y_1,\dots, y_r)$ be a graded pcgs for $L(\phi)$ such that,
   for $L^{(i)} = \langle y_i,\dots, y_r\rangle$, the quotients $L^{(i)} /
   L^{(i+1)}$ are isomorphic to either $\mathbb{Z}$ or $\mathbb{Z}/n$ for
   $n\in\N$. It follows that $L^{(i)} = \langle y_iL^{(i+1)}\rangle$, so let
   $o(y_i)$ be either $\infty$ or $n$, depending on the order of the $L^{(i)} /
   L^{(i+1)}$. For each $a\in L(\phi)$ and $i\in\{1,\dots, r\}$, there exists
   integers $0 \leq k_i(a) < o(y_i)$ such that 
   \begin{equation} \label{eqn:collected-word}
      a = k_1(a) \cdot y_1 + \cdots + k_r(a) \cdot y_r.
   \end{equation}
   We call the expression for $a$ in~\eqref{eqn:collected-word} its normal word
   with respect to $\mathcal{Y}$, which is uniquely determined by $\mathcal{Y}$.
   
   Let $\mathcal{X}=(x_1,\dots, x_r)$ be a sequence of pre-images of
   $\mathcal{Y}$ in $G$ such that $y_i\in L_s(\phi)$ implies that
   $y_ix_i^{-1}\in \partial\phi_s$.  
   Define a function of sets $\pi_{\mathcal{Y}} : L(\phi)\rightarrow G$ such
   that 
   \begin{equation}\label{eqn:pi} 
      a \mapsto x_1^{k_1(a)} \cdots x_r^{k_r(a)}.
   \end{equation}
   Because normal words are unique, $\pi_{\mathcal{Y}}$ is well-defined. By
   Lemma~\ref{lem:weakly-pcgs}, $\mathcal{X}$ contains a pcgs for $G$, so $\pi$
   is surjective.
\end{proof}

\section{Filtered generating sets and lattices}\label{sec:partially-ordered}

Our main objective in this section is to develop a generating set that interacts
nicely with filters. A common theme for using groups effectively in
computational settings is to have a structured generating set for the group.
Some examples include bases and strong generating sets for permutation groups
\cite{Seress:book}*{Chapter~4}, (special) polycyclic generating sequences for
solvable groups \cites{CELG:special, EW:exhibit} and
\cite{Sims:book}*{Chapter~9}, and power-commutator presentations for $p$-groups
\cites{HN:p-quotient,NO:p-quotient2}. These generating sets are all based on a
series in the group. Influenced by these generating sets, we define an
appropriate generating set in the context of filters.

Throughout this section, we fix an $(M, G)$-filter $\phi$. For now, we say a
generating set $\mathcal{X}\subseteq G$ is \emph{filtered} by $\phi$ if
$\mathcal{X}$ contains a generating set for each $\phi_s$ and is compatible with
the induced complete lattice, $\Lat(\phi)$, of $\im(\phi)$, see
Definition~\ref{def:filtered} below. In this section we prove the following
theorem.

\begin{thm}\label{thm:filtered-implications}
   For a group $G$ and a conically pre-ordered monoid $M$, let $\phi$ be an $(M,
   G)$-filter. If the generating set $\mathcal{X}\subseteq G$ is filtered by
   $\phi$, then 
   \begin{enumerate}
      \item[(i)] the complete lattice induced by $\im(\phi)$ is distributive,
      and
      \item[(ii)] $\mathcal{X}$ is filtered by the boundary filter
      $\partial\phi$.
   \end{enumerate}
\end{thm}

Since the lattice of normal subgroups is not in general distributive, Theorem~\ref{thm:filtered-implications} shows that not all filters have such a generating set. We begin with a natural condition on generating sets, akin to strong generating sets from the context of permutation groups.

\begin{definition}\label{def:weakly-filtered}
A set $\mathcal{X}\subseteq G$ is \emph{weakly-filtered} by $\phi$ if for all $s\in M$, $\langle \phi_s\cap \mathcal{X}\rangle= \phi_s$. 
\end{definition}

The property of a generating set $\mathcal{X}$ being weakly-filtered can be
rephrased in the context of partially-ordered sets. Suppose
$\mathcal{X}\subseteq G$ is weakly-filtered by $\phi$. Define functions on
partially-ordered sets $\Nor(G)$ and $2^{\mathcal{X}}$; namely,
$\mathcal{C}:\Nor(G)\rightarrow 2^{\mathcal{X}}$ where $H\mapsto H\cap
\mathcal{X}$ and $\langle\cdot\rangle : 2^{\mathcal{X}} \rightarrow \Nor(G)$
where $Y\mapsto \langle Y\rangle$. These functions are order-preserving because
$H,K \in \Nor(G)$ with $H\leq K$ implies $H\cap \mathcal{X} \subseteq K\cap
\mathcal{X}$, and if $Y,Z\in 2^{\mathcal{X}}$ with $Y\subseteq Z$, then $\langle
Y\rangle\leq \langle Z\rangle$. This proves the following lemma.

\begin{lem}
   If $\mathcal{X}\subseteq G$ is weakly-filtered by $\phi$, then the
   restriction of $\mathcal{C}$ on $\im(\phi)$ is an (order) isomorphism with
   inverse $\langle\cdot\rangle_{\mathcal{X}}:\im(\phi)\cap
   \mathcal{X}\rightarrow \im(\phi)$. 
\end{lem} 

We need a stronger definition for our purposes. The set $\im(\phi)$ is, in
general, not a lattice. Let $\Lat(\phi)\cap \mathcal{X}$ denote the image of
$\Lat(\phi)$ in $2^{\mathcal{X}}$ under $\mathcal{C}$. The surjection
$\mathcal{C}: \Lat(\phi)\rightarrow \Lat(\phi)\cap \mathcal{X}$ is
order-preserving; however, if $H,K\in\Lat(\phi)$ and $H\cap \mathcal{X}\subseteq
K\cap \mathcal{X}$, then $H$ need not be a subgroup of $K$. Even as
partially-ordered sets $\Lat(\phi)$ need not be isomorphic to $\Lat(\phi)\cap
\mathcal{X}$. The strength of the following definition comes when $\mathcal{C}$
and $\langle \cdot\rangle_{\mathcal{X}}$ are complete \emph{lattice}
homomorphisms. 

\begin{definition}\label{def:filtered}
A set $\mathcal{X}\subseteq G$ is \emph{filtered} by $\phi$ if it is weakly-filtered and for all $S\subseteq M$, 
\[\bigcap_{s\in S}\phi_s = \left\langle \bigcap_{s\in S}(\phi_s\cap \mathcal{X})\right\rangle \qquad \text{and} \qquad \left(\prod_{s\in S}\phi_s\right)\cap \mathcal{X} = \bigcup_{s\in S}(\phi_s\cap \mathcal{X}). \]
\end{definition}

This definition is not just a weakly-filtered analogue for the complete lattice $\Lat(\phi)$. 
If $\mathcal{X}$ satisfies the property that for all $H\in\Lat(\phi)$, $\langle H\cap \mathcal{X}\rangle = H$, then $\mathcal{X}$ may still not be filtered by $\phi$. 

\begin{ex}\label{ex:lat-not-weak}
   We consider some subtleties of Definition~\ref{def:filtered}.
Let $G=\mathbb{Z}/60$, and $M=(\N^2,\preceq_+)$.
Define an $(M,G)$-filter $\phi$ where $\phi_0=G$, 
\[ \phi_s = \left\{ \begin{array}{ll} \langle 2 \rangle & \text{if } s=e_1, \\ \langle 3\rangle & \text{if } s=e_2, \\ \langle 10\rangle & \text{if } s=2e_1, \\ \langle 15\rangle & \text{if } s=2e_2, \end{array}\right. \]
and $\phi_s=0$ otherwise. 
Set $\mathcal{X}=\{ 2,3,10,15 \}$, and observe that $\mathcal{X}$ is weakly-filtered by $\phi$. 
In Figure~\ref{fig:Hasse-lattice}, we plot the Hasse diagram of $\im(\phi)$ and $\Lat(\phi)$. 
Set $H=\langle 6\rangle$ and $K=\langle 30\rangle$. Then $H\cap \mathcal{X} = \varnothing = K\cap \mathcal{X}$, but $\langle 6\rangle =H\not\leq K=\langle 30\rangle$.
Thus, $\mathcal{C}:\Lat(\phi)\rightarrow\Lat(\phi)\cap \mathcal{X}$ is not order-preserving and, hence, is not an isomorphism.
By Proposition~\ref{prop:distributive}, $\mathcal{X}$ is not filtered by $\phi$. 

If, instead, we set $\mathcal{X}=\{ 2,3,5,6,10,15,30\}$, then for all $H\in\Lat(\phi)$, $\langle H\cap \mathcal{X}\rangle=H$. 
However, $\mathcal{X}$ is not filtered by $\phi$: for example,
\[ (\phi_{2e_1}\phi_{2e_2})\cap \mathcal{X} = \langle 5\rangle \cap \mathcal{X} = \{ 5, 10, 15, 30\}\] 
and 
\[ (\phi_{2e_1}\cap \mathcal{X})\cup (\phi_{2e_2}\cap \mathcal{X}) = (\langle 10\rangle \cap \mathcal{X}) \cup (\langle 15\rangle \cap \mathcal{X}) = \{ 10, 15, 30\}.\]
If, for example, $\mathcal{X}=\{ 6,10,15,30\}$, then $\mathcal{X}$ is filtered by $\phi$. It is worth pointing out that $\mathcal{X}$ \emph{is} filtered by $\phi$, and it induces a pcgs for the associated Lie ring $L(\phi)$. \qed
\end{ex}

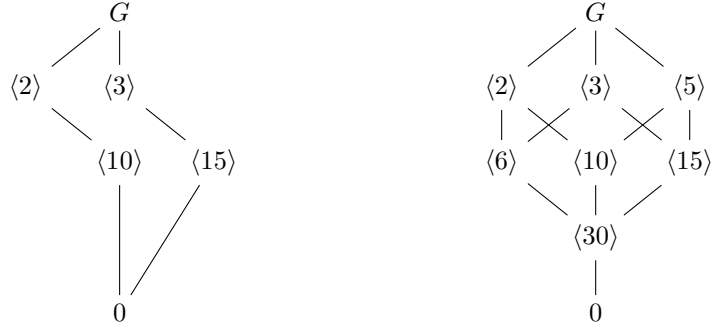
\begin{figure}[ht]
  \centering
  \begin{subfigure}[b]{0.5\textwidth}
    \centering
    \begin{tikzpicture}
      \node(G) at (0,4) {$G$};
      \node(2) at (-1.25,3) {$\langle 2\rangle$};
      \node(3) at (0,3) {$\langle 3\rangle$};
      \node(10) at (0,2) {$\langle 10\rangle$};
      \node(15) at (1.25,2) {$\langle 15\rangle$};
      \node(0) at (0,0) {$0$};
      
      \draw[-] (G) -- (2);
      \draw[-] (G) -- (3);
      \draw[-] (3) -- (15);
      \draw[-] (2) -- (10);
      \draw[-] (10) -- (0);
      \draw[-] (15) -- (0);
    \end{tikzpicture}
    \caption{The Hasse diagram of $\im(\phi)$.}
  \end{subfigure}%
  \begin{subfigure}[b]{0.5\textwidth}
    \centering
    \begin{tikzpicture}
      \node(G) at (0,4) {$G$};
      \node(2) at (-1.25,3) {$\langle 2\rangle$};
      \node(3) at (0,3) {$\langle 3\rangle$};
      \node(5) at (1.25,3) {$\langle 5\rangle$};
      \node(6) at (-1.25,2) {$\langle 6\rangle$};
      \node(10) at (0,2) {$\langle 10\rangle$};
      \node(15) at (1.25,2) {$\langle 15\rangle$};
      \node(30) at (0,1) {$\langle 30\rangle$};
      \node(0) at (0,0) {$0$};
      
      \draw[-] (G) -- (2);
      \draw[-] (G) -- (3);
      \draw[-] (G) -- (5);
      \draw[-] (5) -- (10);
      \draw[-] (5) -- (15);
      \draw[-] (2) -- (6);
      \draw[-] (2) -- (10);
      \draw[-] (3) -- (15);
      \draw[-] (3) -- (6);
      \draw[-] (6) -- (30);
      \draw[-] (10) -- (30);
      \draw[-] (15) -- (30);
      \draw[-] (30) -- (0);
    \end{tikzpicture}
    \caption{The lattice $\Lat(\phi)$.}
  \end{subfigure}
  \caption{Hasse diagrams related to $(M, G)$-filter $\phi$ from Example~\ref{ex:lat-not-weak}.}
  \label{fig:Hasse-lattice}
\end{figure}

If $\mathcal{X}$ is filtered by $\phi$, then $\mathcal{C}$ is an isomorphism,
and therefore, the lattice $\Lat(\phi)$ inherits properties of the subset
lattice $\Lat(\phi)\cap \mathcal{X}$. The next proposition proves
Theorem~\ref{thm:filtered-implications} (i). 

\begin{prop}\label{prop:distributive}
   Let $\phi$ be an $(M, G)$-filter. The set $\mathcal{X}\subseteq G$ is
   filtered by $\phi$ if, and only if, $\mathcal{C}:\Lat(\phi)\rightarrow
   \Lat(\phi)\cap \mathcal{X}$ and $\langle\cdot\rangle_{\mathcal{X}} :
   \Lat(\phi)\cap \mathcal{X}\rightarrow \Lat(\phi)$ are complete lattice
   isomorphisms. In such a case, $\Lat(\phi)$ is a distributive lattice. 
\end{prop}

\begin{proof}
Suppose $\mathcal{X}$ is filtered by $\phi$ and $S\subseteq M$.
Since $\cap$ is associative and $\mathcal{X}$ is filtered by $\phi$,
\begin{align*} 
   \left(\bigcap_{s\in S}\phi_s\right)\cap \mathcal{X} &= \bigcap_{s\in S}(\phi_s\cap \mathcal{X}), & \left(\prod_{s\in S}\phi_s\right)\cap \mathcal{X} &= \bigcup_{s\in S}(\phi_s\cap \mathcal{X}).
\end{align*}
Hence $\mathcal{C}:\Lat(\phi)\rightarrow \Lat(\phi)\cap \mathcal{X}$ is a lattice homomorphism.
Since $\mathcal{X}$ is filtered by $\phi$ it is also weakly-filtered. 
Therefore,  
\begin{align*}
   \left\langle \bigcup_{s\in S}(\phi_s\cap \mathcal{X}) \right\rangle &= \prod_{s\in S}\langle \phi_s\cap \mathcal{X}\rangle = \prod_{s\in S}\phi_s, & \left\langle \bigcap_{s\in S}(\phi_s\cap \mathcal{X})\right\rangle &= \bigcap_{s\in S}\phi_s .
\end{align*}
Therefore, $\langle\cdot\rangle_{\mathcal{X}}:\Lat(\phi)\cap \mathcal{X}\rightarrow \Lat(\phi)$ is a lattice homomorphism.
Both homomorphisms $\mathcal{C}$ and $\langle\cdot\rangle_{\mathcal{X}}$ are order-preserving.
Since $\mathcal{X}$ is weakly-filtered, $\langle\cdot\rangle_{\mathcal{X}}$ is the inverse of $\mathcal{C}$, and hence, $\Lat(\phi)\cong \Lat(\phi)\cap \mathcal{X}$.

Conversely, suppose $\mathcal{C}$ and $\langle\cdot\rangle_{\mathcal{X}}$ are complete lattice isomorphisms.
It follows then that $\mathcal{X}$ is weakly-filtered by $\phi$. 
Let $S\subseteq M$, so $\bigcap_{s\in S}\phi_s\in\Lat(\phi)$.
Since $\langle \cdot\rangle_{\mathcal{X}}$ is a complete lattice homomorphism, 
\[ \bigcap_{s\in S} \phi_s = \left\langle \left(\bigcap_{s\in S}\phi_s\right)\cap \mathcal{X}\right\rangle = \left\langle \bigcap_{s\in S}(\phi_s\cap \mathcal{X})\right\rangle. \]
Furthermore, since $\mathcal{C}$ is a complete lattice homomorphism, 
\[ \left(\prod_{s\in S}\phi_s\right)\cap \mathcal{X} = \bigcup_{s\in S}(\phi_s\cap \mathcal{X}).\qedhere \]
\end{proof}

Now we can prove the second part of Theorem~\ref{thm:filtered-implications} by employing Proposition~\ref{prop:distributive}.
The key to the next proof is to use the fact that $\cap\mathcal{X}$ and $\langle\cdot\rangle_{\mathcal{X}}$ are complete lattice homomorphisms when $\mathcal{X}$ is filtered by $\phi$. 

\begin{proof}[Proof of Theorem~\ref{thm:filtered-implications} (ii)]
Suppose $\mathcal{X}$ is filtered by $\phi$; we will prove that $\mathcal{X}$ is also filtered by $\partial\phi$. 
First we show that for all $S\subseteq M$,
\[ \bigcap_{s\in S} \partial\phi_s = \left\langle \bigcap_{s\in S} (\partial\phi_s\cap\mathcal{X})\right\rangle. \]
By Proposition~\ref{prop:distributive}, $\mathcal{C}$ and $\langle\cdot\rangle$ are complete lattice homomorphisms, so 
\begin{align*}
\left\langle \bigcap_{s\in S}(\partial\phi_s\cap\mathcal{X})\right\rangle = \left\langle \bigcap_{s\in S}\bigcup_{t\in M\setminus\{0\}} (\phi_{s+t}\cap \mathcal{X})\right\rangle = \bigcap_{s\in S} \prod_{t\in M\setminus\{0\}} \langle\phi_{s+t}\cap\mathcal{X}\rangle = \bigcap_{s\in S}\partial\phi_s. 
\end{align*}
For the second part, we show that 
\[ \left(\prod_{s\in S} \partial\phi_s \right) \cap \mathcal{X} = \bigcup_{s\in S} (\partial\phi_s\cap \mathcal{X}). \]
Again, we use the fact that $\cap\mathcal{X}$ is a complete lattice homomorphism:
\[ \left(\prod_{s\in S} \partial\phi_s \right) \cap \mathcal{X} = \left( \prod_{s\in S}\prod_{t\in M\setminus\{0\}} \phi_{s+t} \right) \cap \mathcal{X} = \bigcup_{s\in S} \left( \partial\phi_s \cap\mathcal{X}\right). \]
Therefore, $\mathcal{X}$ is filtered by $\partial\phi$. 
\end{proof}

\section{Faithful filters}\label{sec:faithful-filters}

In this section, we impose one more property on our filters so that the sets
$L(\phi)$ and $G$ are in bijection. The main reason that the surjection
$\pi:L(\phi)\rightarrow G$ from Theorem~\ref{thm:surjection}, cf.\
equation~(\ref{eqn:pi}), might not be injective comes down to the fact that
$(\phi_s\setminus\partial\phi_s)\cap (\phi_t\setminus\partial\phi_t)$ might be
nonempty for distinct $s,t\in M$. Recall the definition of a faithful filter
from Definition~\ref{def:faithful}. Observe that the first property of faithful
filters implies that $\mathcal{X}$ is weakly-filtered by $\phi$, and the second
property implies that $\mathcal{X}$ is filtered by $\phi$. If $\phi$ is a
faithful $(M, G)$-filter such that $\mathcal{X}\subseteq G$ satisfies the three
properties of Definition~\ref{def:faithful}, then we say that $\mathcal{X}$ is
\emph{faithfully filtered} by $\phi$. 

We prove the following theorems in this section.

\begin{thm}\label{thm:faithful-basis}
   Assume $\phi$ is a finite, inertia-free $(M, G)$-filter.
   If $\phi$ is faithful, then every pre-image of every graded pcgs for $L(\phi)$ is filtered by $\phi$. 
\end{thm}

\begin{thm}[Theorem~\ref{thm:Main2}]\label{thm:bijection}  
   If $\phi$ is a finite, faithful, and inertia-free $(M, G)$-filter, then there exists a bijection between $L(\phi)$ and $G$ that maps a pcgs of $L(\phi)$, as an abelian group, to a pcgs of $G$.
\end{thm}

The argument in the next lemma is fundamental to the proofs for the above theorems, and it illustrates a proof by descent for finite inertia-free filters. The essence of the argument is that if $x\in\phi_s\cap \phi_t$, then $x$ is contained in either $\partial\phi_s$ or $\partial\phi_t$ because $\phi$ is faithful. Because $\phi$ is finite and inertia-free, we apply Proposition~\ref{prop:inert} to both $\partial\phi_s$ and $\partial\phi_t$. The element $x$ must be contained in one of these smaller subgroups, say $x\in\phi_u\cap \phi_t$. So we continue descending until we reach a contradiction. 

\begin{lem}\label{lem:boundary-contain}
   Suppose $\phi$ is a finite, faithful, and inertia-free $(M, G)$-filter. If
   $s, t\in M$ such that $\phi_s < \phi_t$, then $\phi_s\leq \partial \phi_t$. 
\end{lem}

\begin{proof}
   If $\partial\phi_t=\phi_t$, then we are done. Otherwise, $L_t(\phi)\ne 0$.
   Since $\phi$ is faithful, there exists $\mathcal{X}\subseteq G$ such that
   $\mathcal{X}$ is faithfully filtered by $\phi$. Without loss of generality,
   we assume that $1\notin\mathcal{X}$. Let $\mathcal{X}_t =
   (\phi_t\setminus\partial\phi_t) \cap \mathcal{X}$. Since $\mathcal{X}$ is
   filtered by $\phi$, by Theorem~\ref{thm:filtered-implications} $\langle
   \mathcal{X}_t\rangle \partial\phi_t = \phi_t$. If $\phi_s\cap
   \mathcal{X}_t=\emptyset$, then by Theorem~\ref{thm:filtered-implications} and
   Proposition~\ref{prop:distributive}, 
   \begin{align*}
      \phi_s = \langle \phi_s\cap\phi_t\cap \mathcal{X} \rangle \leq \langle \phi_s \cap \mathcal{X}_t \rangle \partial\phi_t = \partial\phi_t.
   \end{align*}
   Otherwise, $\phi_s\cap\mathcal{X}_t\ne \emptyset$. Since $\mathcal{X}$ is
   faithfully filtered by $\phi$, 
   \[ 
      \phi_s\cap\mathcal{X}_t = \partial\phi_s\cap\mathcal{X}_t\ne \emptyset. 
   \]

   Let $x\in\partial\phi_s\cap\mathcal{X}_t$. Since $\phi$ is finite and
   inertia-free, there exists $I_s\subseteq \mathcal{I}_\phi$ such that
   $\partial\phi_s = \langle \phi_u\mid u\in I_s\rangle$ by
   Proposition~\ref{prop:inert}. Since $\mathcal{X}$ is filtered by $\phi$, 
   \[ 
      \partial\phi_s\cap\mathcal{X}=\bigcup_{u\in I_s} (\phi_u\cap\mathcal{X}). 
   \]
   For each $u\in I_s$, $\phi_u < \phi_s$, and since $x\in\mathcal{X}_t$, there
   exists $u\in I_s$ such that $x\in \phi_u$. Now we are back to where we were
   earlier but with a new subgroup: $\phi_u < \phi_t$ and $x\in\phi_u\cap
   \mathcal{X}_t$. Because $\phi$ is finite and inertia-free, $x\in\bigcap_{s\in
   M}\phi_s\cap \mathcal{X}$, so $x=1$, which is a contradiction. Therefore,
   $\phi_s\cap\mathcal{X}_t\ne\emptyset$ cannot happen, so the lemma follows. 
\end{proof}

Recall that $\parallel$ denotes that two elements of a partially-ordered set are
incomparable. 

\begin{prop}\label{prop:intersection}
   Suppose $\phi$ is a finite, faithful, and inertia-free $(M, G)$-filter, and suppose $S\subseteq M$ such that $|S|\geq 2$.
   If for every distinct pair $s,t\in S$, $\phi_s\parallel\phi_t$, then 
   \[ \bigcap_{s\in S}\phi_s = \bigcap_{s\in S}\partial\phi_s. \]
\end{prop}

\begin{proof}
   Since $\phi_s\geq \partial\phi_s$ for all $s\in M$, we need only show one containment direction. If $\bigcap_{s\in S}\phi_s = 1$, then we are done, so suppose that $\bigcap_{s\in S}\phi_s\ne 1$. Since $\phi$ is faithful, there exists $\mathcal{X}\subseteq G$ faithfully filtered by $\phi$. Thus, $\bigcap_{s\in S} \phi_s= \langle \bigcap_{s\in S} (\phi_s\cap\mathcal{X})\rangle$. Since the intersection is nontrivial, it follows there exists an $x\ne 1$ such that 
   \begin{align}\label{eqn:little-x}
      x\in \left(\bigcap_{s\in S}\phi_s\right)\cap \mathcal{X} = \bigcap_{s\in S}(\phi_s\cap\mathcal{X}).
   \end{align}
   Since $\phi$ is faithful and $x\in \mathcal{X}$, there exists a unique $t\in M$ such that $x\in\phi_t\setminus\partial\phi_t$. If $t\notin S$, then $x\in \bigcap_{s\in S}\partial\phi_s$. If this holds for all choices of $x$ that satisfy~\eqref{eqn:little-x}, then 
   \begin{align*}
      \bigcap_{s\in S}\phi_s = \left\langle \bigcap_{s\in S}(\phi_s\cap\mathcal{X}) \right\rangle \leq \bigcap_{s\in S}\partial\phi_s. 
   \end{align*}
   Therefore, the lemma follows in this case.
   
   Otherwise, assume $t\in S$. Since $\phi$ is faithful, for every $s\in S\setminus\{t\}$, $x\in\partial\phi_s \cap (\phi_t\setminus\partial\phi_t)$. Because $|S|\geq 2$, there is at least one such $s$. As $\phi$ is finite and inertia-free, it follows that there exists $I_s\subseteq \mathcal{I}_\phi$ such that $\partial\phi_s = \langle \phi_u\mid u\in I_s\rangle$ by Proposition~\ref{prop:inert}. 
   For $u\in I_s$, we consider three cases: $\phi_u\geq \phi_t$, $\phi_u< \phi_t$, and $\phi_u\parallel \phi_t$. We show that all cases lead to a contradiction. And therefore, $t\notin S$. 
   
   Since $\phi_s\parallel\phi_t$, there cannot be $u\in I_s$ such that $\phi_u\geq\phi_t$. This would imply that $\phi_s \geq \partial\phi_s \geq \phi_u \geq \phi_t$. By the same argument as used in the proof of Lemma~\ref{lem:boundary-contain} (using the fact that $\mathcal{X}$ is filtered by $\phi$), there exists $u\in I_s$ such that $x\in\phi_u$. If $\phi_u < \phi_t$, then by Lemma~\ref{lem:boundary-contain}, $\phi_u\leq \partial\phi_t$, which would be a contradiction since $x\notin\partial\phi_t$. Thus, we have shown that $\phi_u\parallel\phi_t$. Since $u\in\mathcal{I}_\phi$, $\partial\phi_u\ne\phi_u$, and we can continue this descent down the $\mathfrak{B}$-chain. Eventually, we reach a subgroup that must be contained in $\phi_t$; in particular, $\mathfrak{B}_0=\{\langle 1\rangle\}$. This is a contradiction, and so we cannot have $t\in S$. 
\end{proof}
 
From the above statements, faithful filters are highly structured filters. Indeed, if $\phi$ is a faithful filter and $\mathcal{X}$ is filtered by $\phi$, then $\mathcal{X}$ is faithfully filtered by $\phi$.

\subsection{Proof of Theorem~\ref{thm:faithful-basis}}

Now we are ready to prove that every graded pcgs of $L(\phi)$, for a finite, faithful, inertia-free filter $\phi$, induces a faithfully filtered set of $G$. 
The following proof uses Noetherian induction, going up the $\mathfrak{B}$-sequence 
\[ \{\langle 1\rangle \} = \mathfrak{B}_0 \subseteq \mathfrak{B}_1 \subseteq \cdots. \]

\begin{definition}
   For an $(M, G)$-filter $\phi$, a set $\mathcal{X}\subseteq G$ is \emph{$\mathfrak{B}_n$-filtered} if $\mathcal{X}$ is weakly-filtered by $\phi$ and for all $S\subseteq M$ such that $\{ \phi_s \mid s\in S\}\subseteq \mathfrak{B}_n$, 
   \begin{align*}
      \bigcap_{s\in S} \phi_s &= \left\langle \bigcap_{s\in S} (\phi_s \cap\mathcal{X}) \right\rangle & \left(\prod_{s\in S}\phi_s\right) \cap \mathcal{X} &= \bigcup_{s\in S}(\phi_s \cap \mathcal{X}) . 
   \end{align*}
\end{definition}

The idea is to assume that a pre-image $\mathcal{X}$ of an arbitrary graded pcgs $\mathcal{Y}$ of $L(\phi)$ is filtered by $\phi$ up to some $\mathfrak{B}_n$. 
Because $\langle 1\rangle \in\im(\phi)$, this holds for $\mathfrak{B}_0$. 
Then for every group $\phi_s\in\mathfrak{B}_{n+1}$, there exists $B\in\mathfrak{B}_n$ such that $\partial\phi_s=\langle B\rangle$. 
Thus, $\partial\phi_s$ is handled by the induction hypothesis, and all that is left are quotients $\phi_s/\partial\phi_s=L_s(\phi)$.

In the next lemma, we denote the complete lattice induced on $\mathfrak{B}_n$ by $\Lat(\mathfrak{B}_n)$. As $\mathfrak{B}_n\subseteq\im(\phi)$, it follows that $\Lat(\mathfrak{B}_n)$ is a sublattice of $\Lat(\phi)$. 

\begin{lem}\label{lem:B-filtered}
   Let $\phi$ be a finite inertia-free $(M, G)$-filter. The set $\mathcal{X}\subseteq G$ is $\mathfrak{B}_n$-filtered if, and only if, $\mathcal{X}$ is weakly-filtered by $\phi$ and $\left.\mathcal{C}\right|_{\mathfrak{B}_n} : \Lat(\mathfrak{B}_n)\rightarrow \Lat(\mathfrak{B}_n)\cap \mathcal{X}$ and $\langle\cdot\rangle_{\mathcal{X}} : \Lat(\mathfrak{B}_n)\cap\mathcal{X} \rightarrow \Lat(\mathfrak{B}_n)$. 
\end{lem}

\begin{proof}
   The proof follows from the proof of Proposition~\ref{prop:distributive}.
\end{proof}

\begin{proof}[Proof of Theorem~\ref{thm:faithful-basis}]
   Let $\mathcal{Y}$ be a graded pcgs for $L(\phi)$ whose pre-image in $G$ is
   denoted by $\mathcal{X}$. By Lemma~\ref{lem:weakly-pcgs}, $\mathcal{X}$ is
   weakly-filtered by $\phi$. It follows that $\mathcal{X}$ is
   $\mathfrak{B}_0$-filtered, so suppose $n\geq 0$ and $\mathcal{X}$ is
   $\mathfrak{B}_n$-filtered. 
   
   Let $S\subseteq M$ such that $\{\phi_s\mid s\in S\}\subseteq \mathfrak{B}_{n+1}$. Using Lemma~\ref{lem:boundary-contain}, we assume, without loss of generality, that $|S|\geq 2$ and for all distinct $s, t\in S$, $\phi_s\parallel\phi_t$. 
   First we show that
   \begin{equation}\label{eqn:induction-join} 
      \left( \prod_{s\in S} \phi_s \right)\cap \mathcal{X} = \bigcup_{s\in S}(\phi_s\cap\mathcal{X}).
   \end{equation}
   Since $\mathcal{X}$ is weakly-filtered by $\phi$, the ``$\supseteq$''-containment of~\eqref{eqn:induction-join} holds. Thus, we just prove the ``$\subseteq$''-containment. Let $H=\prod_{s\in S}\phi_s$ and $K=\prod_{s\in S}\partial\phi_s$. Then $ H\cap \mathcal{X} = ((H\setminus K)\cap \mathcal{X}) \cup (K \cap\mathcal{X})$. By induction, 
   \[K\cap \mathcal{X} = \bigcup_{s\in S}(\partial\phi_s\cap \mathcal{X})\subseteq \bigcup_{s\in S}(\phi_s\cap \mathcal{X}). \]
   If $H=K$, then we are done, so suppose $H\ne K$. Proposition~\ref{prop:intersection} implies that $H/K \cong \bigoplus_{s\in S}L_s(\phi)$. This implies that there exists $x\in (H\setminus K)\cap\mathcal{X}$ since $\mathcal{X}$ is a pre-image of the pcgs $\mathcal{Y}$ for $L(\phi)$. Let $y\in \mathcal{Y}$ be the element corresponding to $x$. Since $\mathcal{Y}$ is a graded pcgs for $L(\phi)$, there exists a unique $t\in M$ such that $y\in L_t(\phi)$. Thus, $x\in\phi_t\setminus\partial\phi_t$, so $t\in S$. Therefore, $x\in(\phi_t\cap \mathcal{X}) \subseteq \bigcup_{s\in S}(\phi_s\cap\mathcal{X})$. Therefore,~\eqref{eqn:induction-join} holds. 

   Now we show that the intersection equality holds; namely
   \begin{align}\label{eqn:induction-meet}
      \bigcap_{s\in S} \phi_s = \left\langle \bigcap_{s\in S} (\phi_s \cap\mathcal{X}) \right\rangle.
   \end{align}
   Since $\phi_s\in\mathfrak{B}_{n+1}$, there exists $B_s\subseteq\mathfrak{B}_n$ such that $\partial\phi_s = \langle B_s\rangle$. Therefore, 
   \begin{align*}
      \left\langle \bigcap_{s\in S}\phi_s\cap \mathcal{X}\right\rangle &= \left\langle \bigcap_{s\in S}\partial\phi_s\cap \mathcal{X}\right\rangle & (\text{Proposition~\ref{prop:intersection}}) \\
      &= \left\langle \bigcap_{s\in S}\left(\prod_{H\in B_s}H\right)\cap \mathcal{X}\right\rangle & \left(\begin{array}{c} B_s\in\mathfrak{B}_{n} \\ \partial\phi_s = \langle B_s\rangle \end{array}\right)\\
      &= \left\langle \bigcap_{s\in S}\bigcup_{H\in B_s} (H\cap \mathcal{X}) \right\rangle & (\text{equation~\eqref{eqn:induction-join}})\\
      &= \bigcap_{s\in S}\prod_{H\in B_s} \langle H\cap \mathcal{X}\rangle & \left(\begin{array}{c}\text{induction and} \\ \text{Lemma~\ref{lem:B-filtered}}\end{array}\right) \\
      &= \bigcap_{s\in S}\partial\phi_{s} & (\text{weakly-filtered}) \\
      & = \bigcap_{s\in S}\phi_s . & (\text{Proposition~\ref{prop:intersection}})
   \end{align*}
   Therefore, $\mathcal{X}$ is $\mathfrak{B}_{n+1}$-filtered. Since $\phi$ is finite and inertia-free, $\mathcal{X}$ is filtered by $\phi$. Since $\mathcal{X}$ is a pre-image of a graded pcgs, $\mathcal{X}$ is faithfully filtered by $\phi$. 
\end{proof}

Example~\ref{ex:infinite} illustrates one instance of a filter that is not faithful. 
This problem seems to comes up naturally, and this phenomenon arises again in an example in Section~\ref{sec:examples-section}.
It is not known if a method exists in general to address the issue in Example~\ref{ex:infinite} similarly to the way in which Theorem~\ref{thm:nu-filter} addresses Example~\ref{ex:trivial}. 

\subsection{Proof of Theorem~\ref{thm:Main2}}

The crux of Theorem~\ref{thm:Main2} is not the bijection between $G$ and $L(\phi)$---though that is necessary for our purposes---the main point is actually the induced bijection between graded pcgs of $L(\phi)$ and pcgs of $G$ filtered by $\phi$. 
This is critical to proving Theorem~\ref{thm:Main3}.
Consider Example~\ref{ex:just-bijection}. There, $L(\phi)$ and $G$ are in bijection, but the bijection does not help us lift isomorphisms of graded Lie rings to group isomorphisms. 

\begin{proof}[Proof of Theorem~\ref{thm:Main2}]
   Let $\mathcal{Y} = (y_1,\dots, y_r)$ be a minimal graded pcgs for $L(\phi)$ and $\mathcal{X}=(x_1,\dots, x_r)$ a pre-image of $\mathcal{Y}$ in $G$. 
   We use the same notation from the proof of Theorem~\ref{thm:Main}, see~\eqref{eqn:pi} for details.
   By Theorem~\ref{thm:faithful-basis}, $\mathcal{X}$ is faithfully filtered by $\phi$. 
   From the proof of Theorem~\ref{thm:surjection}, the map the $\pi_{\mathcal{Y}} : L(\phi)\rightarrow G$ is a surjection of sets.

   By Lemma~\ref{lem:weakly-pcgs}, $\mathcal{X}$ contains a pcgs of $G$.
   Suppose for some $x\in \mathcal{X}$, the set $\mathcal{X}\setminus\{x\}$ still contains a pcgs for $G$. Let $y$ denote the element corresponding to $x$ in $\mathcal{Y}$. 
   Since $\mathcal{X}$ is faithfully filtered by $\phi$, for each $i\in\{1,\dots, r\}$, there exists a unique $s_i\in M$ such that $x_i \in\phi_{s_i}\setminus\partial\phi_{s_i}$ and $y_i\in L_{s_i}(\phi)$. In particular, there exists unique $s\in M$ such that $x\in\phi_s\setminus\partial\phi_s$. 
   Let $\{x_{i_1},\dots, x_{i_n}\} = (\phi_s\cap \mathcal{X}) \setminus (\{x\} \cup \partial\phi_s)$, so there exists integers $e_{i_j}$ such that 
   \begin{align}\label{eqn:x-word}
      x^{-1}\left(x_{i_1}^{e_{i_1}} \cdots x_{i_n}^{e_{i_n}}\right) &\in\partial\phi_s.
   \end{align}
   If $y_{i_j}$ is the element of $\mathcal{Y}$ corresponding to $x_{i_j}$, then~\eqref{eqn:x-word} implies that in $L_s(\phi)$: 
   \begin{align}\label{eqn:lin-dep}
      y = e_{i_1}\cdot y_{i_1} + \cdots + e_{i_n}\cdot y_{i_n}.
   \end{align}
   However, by minimality $\mathcal{Y}\setminus\{y\}$ is not a pcgs of
   $L(\phi)$, so~\eqref{eqn:lin-dep} implies a contradiction. Hence,
   $\mathcal{X}$ is a pcgs for $G$, and therefore, every $g\in G$ is expressed
   by a unique normal word in $\mathcal{X}$. Therefore, $\pi$ is injective.
\end{proof}

\subsection{Proof of Theorem~\ref{thm:Main3}}

\begin{proof}
   Suppose $\alpha:G\rightarrow H$ is an isomorphism, and let $\phi$ be a
   faithful $(M,G)$-filter. Therefore, $\theta := \phi^\alpha$ is a faithful
   $(M, H)$-filter, and $\alpha$ induces an $M$-graded Lie isomorphism
   $\hat{\alpha}:L(\phi)\rightarrow L(\theta)$. In particular $\hat{\alpha}$
   maps a pcgs, $\mathcal{Y}$, for $L(\phi)$ to a pcgs for $L(\theta)$. By
   Theorem~\ref{thm:Main2}, every pre-image $\mathcal{X}$ of $\mathcal{Y}$ is a
   pcgs for $G$, so we fix one pre-image. Set $\mathcal{X}_s = (\phi_s \setminus
   \partial\phi_s) \cap \mathcal{X}$ for all $s\in S$. Since $\phi$ is faithful,
   we identify the disjoint union with the union:  
   \[ 
      \bigsqcup_{s\in \mathcal{I}_{\phi}} \mathcal{X}_s 
      = \bigcup_{s\in\mathcal{I}_{\phi}} \mathcal{X}_s = \mathcal{X}. 
   \] 
   Thus, there exists a transversal $\sigma_s : L_s(\theta) \rightarrow
   \theta_s$ such that the partial lift $\psi : \mathcal{X} \rightarrow H$,
   mapping $x\in\mathcal{X}_s$ to $(\partial\phi_sx)^{\beta\tau_s}$, is equal to
   the restriction of $\alpha$ on $\mathcal{X}$. Since $\mathcal{X}$ is a pcgs
   for $G$, it follows that $\psi$ induces $\alpha$. Hence, every isomorphism
   $\alpha:G\rightarrow H$ is realized as a lift of an $M$-graded Lie ring
   isomorphism. 
\end{proof}

\section{Examples}\label{sec:examples-section}

We close with two examples that demonstrate the desire for filters to break away
from the constraints of totally ordered monoids since incorporating more
complicated characteristic structure provides a more significant computational
savings. The examples also demonstrate the potential challenge of guaranteeing
the faithful property in filters. While Theorem~\ref{thm:general-progressive}
constructs inertia-free filters, we do not yet have a general construction for
faithful, inertia-free filters. 

Recall that for a group $G$, $\gamma_1=G$ and $\gamma_{k+1} = [\gamma_k, G]$ for
$k\geq 1$.

\begin{ex}\label{ex:UT4}

Let $G$ be the group of $4\times 4$ upper unitriangular matrices over
$\mathbb{Z}$. We define two characteristic subgroups of $G$:
\begin{align*} 
   H &= \begin{bmatrix} 
      1 & * & * & * \\ 
      & 1 & 0 & * \\ 
      & & 1 & * \\ 
      & & & 1 
   \end{bmatrix}, & 
   K &= \begin{bmatrix} 
      1 & 0 & * & * \\ 
      & 1 & * & * \\ 
      & & 1 & 0 \\ 
      & & & 1 
   \end{bmatrix}.
\end{align*}

Let $M=(\N^2, \preceq_+)$, and set $e_1=(1,0)$ and $e_2=(0,1)$. Since $G=HK$, we
define $\phi_{e_1}=H$ and $\phi_{e_2}=K$ in our $(M, G)$-filter $\phi$. We apply
the generating formula from~\cite{W:char}*{Theorem~3.3} to generate a filter
from this data---together with $\phi_0=G$. We plot $\phi$ in
Figure~\ref{fig:UT4}. As abelian groups, $L(\phi)\cong \mathbb{Z}^6$, and the
$(M, G)$-filter is faithful and inertia-free, with $\mathcal{X}=\{1+E_{ij} \mid
1 \leq i < j \leq 4\}$, where $E_{ij}$ is the matrix with a 1 in the $(i, j)$
entry and 0 elsewhere. 

Observe in Figure~\ref{fig:UT4} that $\phi_{(i,j)} = \langle 1\rangle$ whenever
$i\geq 2$ or $j\geq 3$. Instead of the infinite monoid $\N^2$, we can define
$\phi$ over the finite monoid $(C_{3, 1}\times C_{2, 1},\preceq_+)$. However, we
will choose a different finite monoid. Since $G$ is class 3, we define a
congruence for $\N^2$ that identifies all sums of at least $4$ nontrivial
elements as the same element in our monoid. This models the fact that
$\gamma_4(G)=1$. 

Define a congruence $\sim$ on $\N^2$ as follows
\begin{align*}
   (i, j) \sim (k, \ell) \iff \left\{ \begin{array}{l}
      (i, j) = (k, \ell), \text{ or}\\
      i+j \geq 4 \text{ and } k+\ell \geq 4. 
   \end{array}\right.
\end{align*}
Let $M'=(\N^2/\!\sim, \preceq_+)$, and let $\phi'$ be an $(M', G)$-filter as
seen in Figure~\ref{fig:UT4-2}. One can define $\phi'$ similar to $\phi$, but
$\phi'$ has the property that if $i+j=k>0$, then $\gamma_{k+1}(G) <
\phi_{(i,j)}' \leq \gamma_k(G)$. Since $\phi'$ is not faithful, it 

\begin{figure}[ht]
   \begin{subfigure}[b]{0.45\textwidth}
      \centering
      \begin{tikzpicture}
         \pgfmathsetmacro{\myscale}{0.6}
         \node (ptx) at (-0.328,-0.2) {};
         \node (pty) at (-0.2,-0.328) {};
         \node (x) at (0.4 + 3*\myscale,-0.2) {};
         \node (y) at (-0.2, 0.45 + 2*\myscale) {};
         \draw[->] (ptx) -- (x);
         \draw[->] (pty) -- (y);
      
         \node (x0) at (0,-0.5) {0};
         \node (x1) at (\myscale,-0.5) {1};
         \node (x2) at (2*\myscale,-0.5) {2};
         \node (x3) at (3*\myscale,-0.5) {3};
         \node (y0) at (-0.5,0) {0};
         \node (y1) at (-0.5,\myscale) {1};
         \node (y2) at (-0.5,2*\myscale) {2};
         
         \node (G) at (0,0) {$G$};
         \node (H) at (\myscale,0) {$H$};
         \node (G3) at (2*\myscale,0) {$\gamma_3$};
         \node (G4) at (3*\myscale,0) {$1$};
         \node (K) at (0,\myscale) {$K$};
         \node (G2) at (\myscale,\myscale) {$\gamma_2$};
         \node (G32) at (2*\myscale,\myscale) {$\gamma_3$};
         \node (G42) at (3*\myscale,\myscale) {$1$};
         \node (L) at (0,2*\myscale) {$1$};
         \node (G33) at (\myscale,2*\myscale) {$1$};
         \node (G43) at (2*\myscale,2*\myscale) {$1$};
         \node (G44) at (3*\myscale,2*\myscale) {$1$};
      \end{tikzpicture}
      \caption{The $(M, G)$-filter $\phi$.}
      \label{fig:UT4}
   \end{subfigure}\hfill%
   \begin{subfigure}[b]{0.45\textwidth}
      \centering
      \begin{tikzpicture}
         \pgfmathsetmacro{\myscale}{0.5}
         \node (ptx) at (-0.328,-0.2) {};
         \node (pty) at (-0.2,-0.328) {};
         \node (x) at (0.4 + 4*\myscale,-0.2) {};
         \node (y) at (-0.2, 0.45 + 4*\myscale) {};
         \draw[->] (ptx) -- (x);
         \draw[->] (pty) -- (y);
      
         \node (x0) at (0,-0.5) {0};
         \node (x1) at (\myscale,-0.5) {1};
         \node (x2) at (2*\myscale,-0.5) {2};
         \node (x3) at (3*\myscale,-0.5) {3};
         \node (x4) at (4*\myscale,-0.5) {4};
         \node (y0) at (-0.5,0) {0};
         \node (y1) at (-0.5,\myscale) {1};
         \node (y2) at (-0.5,2*\myscale) {2};
         \node (y3) at (-0.5,3*\myscale) {3};
         \node (y4) at (-0.5,4*\myscale) {4};
         
         \node (G) at (0,0) {$G$};
         \node (H) at (\myscale,0) {$H$};
         \node (G3) at (2*\myscale,0) {$\gamma_2$};
         \node (G4) at (3*\myscale,0) {$\gamma_3$};
         \node (G4) at (4*\myscale,0) {$1$};
         \node (K) at (0,\myscale) {$K$};
         \node (G2) at (\myscale,\myscale) {$\gamma_2$};
         \node (G32) at (2*\myscale,\myscale) {$\gamma_3$};
         \node (G42) at (3*\myscale,\myscale) {$1$};
         \node (L) at (0,2*\myscale) {$\gamma_2$};
         \node (G33) at (\myscale,2*\myscale) {$\gamma_3$};
         \node (G43) at (2*\myscale,2*\myscale) {$1$};
         \node at (0,3*\myscale) {$\gamma_3$};
         \node at (1*\myscale,3*\myscale) {$1$};
         \node at (0,4*\myscale) {$1$};
      \end{tikzpicture}
      \caption{The $(M', G)$-filter $\phi'$.}
      \label{fig:UT4-2}
   \end{subfigure}
   \caption{The plots of two filters from Example~\ref{ex:UT4}. The filter in (A) is inertia-free and faithful, but the filter in (B) is only inertia-free.}
   \label{fig:UT4-main}
\end{figure}

\end{ex}

\begin{ex}\label{ex:genus3}
   We consider a group examined in \cite{ELGO:auts}*{Section 12.1} and \cite{M:efficient-filters}*{Section 5}.
   For a fixed odd prime $p$, we define a $p$-group $G$ by a power-commutator presentation, where all trivial commutators are omitted
   \begin{align*} 
   G = \langle g_1,\dots,g_{13} &\mid \text{exponent }p, [g_{10},g_6]=g_{11}, [g_{10},g_7] = g_{12},\\
   &\;\;\; [g_2,g_1]=[g_4,g_3]=[g_6,g_5]=[g_8,g_7]=[g_{10},g_9]=g_{13}\rangle.
   \end{align*}
   In \cite{M:efficient-filters}, we defined an $(\N^2, G)$-filter, $\tau$, where $\N^2$ is totally ordered by $\preceq_\ell$. 

   Observe from the presentation that $G$ has class 2 and $\gamma_2= \langle g_{11},g_{12},g_{13}\rangle$. 
   The following subgroups are characteristic
   \begin{align*}
   J_1 &= \langle g_1,\dots,g_9,\gamma_2\rangle, & J_4 &= \langle g_9,\gamma_2 \rangle,\\
   J_2 &= \langle g_1,\dots,g_5,g_8,g_9,\gamma_2 \rangle, & H &= \langle g_{13}\rangle. \\
   J_3 &= \langle g_5,g_8,g_9,\gamma_2\rangle, & &
   \end{align*}
   The details of this are given in \cite{M:efficient-filters}.
   The image of $\tau$ produces the following characteristic series
   \[ G> J_1 > J_2 > J_3 > J_4 > \gamma_2 > H > 1.\]

   Using techniques developed in \cite{BW:isometry}, $G$ has more characteristic
   subgroups:
   \begin{align*}
   K_1 &= \langle g_5,\dots, g_{10}, \gamma_2 \rangle, & 
   K_2 &= \langle g_1,\dots,g_4,\gamma_2 \rangle.
   \end{align*}
   Let $M=\N^2\times \N\times \N$, where $M$ is ordered by the direct product ordering: for $(s, i, j), (t, k, l)\in M$, $(s, i, j)\preceq (t, k, l)$ if $s\preceq_\ell t$, $i\leq k$ and $j\leq l$. 
   Let 
   \begin{align*} 
      D &= \{(s, 0, 0)\in M \mid s\in \N^2 \}\cup \{e_2, e_3\}, &
      E &= M\setminus D.
   \end{align*}
   We define a function, $\pi$, on $D$ into $\Nor(G)$ via $\pi_{(s, 0, 0)} =
   \tau_s$, $\pi_{e_2} = K_1$, and $\pi_{e_3}=K_2$. We define an $(M, G)$-filter
   $\phi$ such that for $s\in M$, 
   \begin{align*}
      \phi_s &= \prod_{\textbf{r}\in R_E(s)} [\pi_{\textbf{r}}].
   \end{align*}
   From \cite{W:char}*{Theorem~3.3}, $\phi$ is an $(M, G)$-filter. We plot
   $\phi$ in Figure~\ref{fig:genus3} along with its lattice $\Lat(\phi)$. With
   just these characteristic subgroups, the potential order of the Lie
   automorphism group has decreased from roughly $p^{10^2}$ to roughly $p^{39}$,
   the order of the group stabilizing the arrangement of subspaces in
   $G/\gamma_2$, see Figure~\ref{fig:filter-lattice}. 
   
   Since $\tau$ is inertia-free, $\phi$ is inertia-free. If
   $\mathcal{X}=\{g_1,\dots, g_{13}\}$, then $\mathcal{X}$ is filtered by
   $\phi$. However, $\phi$ is not a faithful filter, which can be seen in
   Figure~\ref{fig:genus3}. One can define a new filter $\theta$ from $\phi$
   such that $\theta$ is faithful and inertia-free, but it is not known how to
   do this in general. It is also not known how this affects the associated Lie
   ring. Of course Theorem~\ref{thm:Main2} states they are in bijection, but it
   is not clear if, for example, the center of the Lie ring is larger than the
   center of the group. In addition, $\phi$ was defined arbitrarily from $\tau$.
   It is not clear if there is a ``best'' way to refine a filter. 

   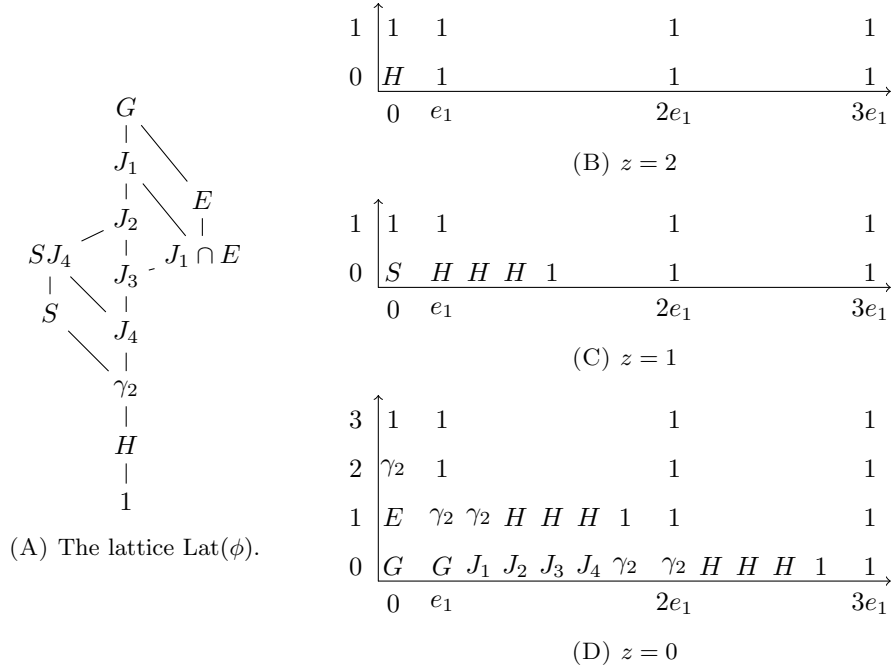
\begin{figure}[ht]
      \begin{subfigure}{0.32\textwidth}
         \centering
         \begin{tikzpicture}[scale=0.75]
            \node (G) at (0,7) {$G$};
            \node (J1) at (0,6) {$J_1$};
            \node (J2) at (0,5) {$J_2$};
            \node (J1capE) at (1.6,4.35) {$J_1\cap K_1$};
            \node (J3) at (0,4) {$J_3$};
            \node (J4) at (0,3) {$J_4$};
            \node (G2) at (0,2) {$\gamma_2$};
            \node (H) at (0,1) {$H$};
            \node (1) at (0,0) {$1$};
            \node (S) at (-1.5,3.35) {$K_2$};
            \node (SJ4) at (-1.5,4.35) {$K_2J_4$};
            \node (E) at (1.6,5.35) {$K_1$};

            \draw[-] (G) -- (J1);
            \draw[-] (G) -- (E);
            \draw[-] (J1) -- (J1capE);
            \draw[-] (J1) -- (J2);
            \draw[-] (J2) -- (J3);
            \draw[-] (E) -- (J1capE);
            \draw[-] (J1capE) -- (J3);
            \draw[-] (J2) -- (SJ4);
            \draw[-] (J3) -- (J4);
            \draw[-] (J4) -- (G2);
            \draw[-] (S) -- (SJ4);
            \draw[-] (S) -- (G2);
            \draw[-] (SJ4) -- (J4);
            \draw[-] (G2) -- (H);
            \draw[-] (H) -- (1);
         \end{tikzpicture}
         \caption{The lattice $\Lat(\phi)$.}
         \label{fig:filter-lattice}
      \end{subfigure}\hfill%
      \begin{subfigure}{0.65\textwidth}
         \centering
         \begin{subfigure}{\textwidth}
            \centering 
            \begin{tikzpicture}
               \pgfmathsetmacro{\myscale}{0.65}
               \node (ptx) at (-0.328,-0.2) {};
               \node (pty) at (-0.2,-0.328) {};
               \node (x) at (0.4 + 9.75*\myscale,-0.2) {};
               \node (y) at (-0.2, 0.45 + \myscale) {};
               \draw[->] (ptx) -- (x);
               \draw[->] (pty) -- (y);
            
               \node (x0) at (0,-0.5) {0};
               \node (x1) at (\myscale,-0.5) {$e_1$};
               \node (x5) at (5.75*\myscale,-0.5) {$2e_1$};
               \node (x5) at (9.75*\myscale,-0.5) {$3e_1$};
               \node (y0) at (-0.5,0) {0};
               \node (y1) at (-0.5,\myscale) {1};
               
               \node (G) at (0,0) {$H$};
               \node (H) at (\myscale,0) {$1$};
               \node (G40) at (5.75*\myscale,0) {$1$};
               \node (G40) at (9.75*\myscale,0) {$1$};
               \node (G) at (0,\myscale) {$1$};
               \node (H) at (\myscale,\myscale) {$1$};
               \node (G40) at (5.75*\myscale,\myscale) {$1$};
               \node (G40) at (9.75*\myscale,\myscale) {$1$};
            \end{tikzpicture}
            \caption{$z=2$}
            \label{fig:z=2}
         \end{subfigure}
         \begin{subfigure}{\textwidth}
            \centering 
            \begin{tikzpicture}
               \pgfmathsetmacro{\myscale}{0.65}
               \node (ptx) at (-0.328,-0.2) {};
               \node (pty) at (-0.2,-0.328) {};
               \node (x) at (0.4 + 9.75*\myscale,-0.2) {};
               \node (y) at (-0.2, 0.45 + \myscale) {};
               \draw[->] (ptx) -- (x);
               \draw[->] (pty) -- (y);
            
               \node (x0) at (0,-0.5) {0};
               \node (x1) at (\myscale,-0.5) {$e_1$};
               \node (x5) at (5.75*\myscale,-0.5) {$2e_1$};
               \node (x5) at (9.75*\myscale,-0.5) {$3e_1$};
               \node (y0) at (-0.5,0) {0};
               \node (y1) at (-0.5,\myscale) {1};
               
               \node (G) at (0,0) {$K_2$};
               \node (H) at (\myscale,0) {$H$};
               \node (G3) at (1.75*\myscale,0) {$H$};
               \node (G4) at (2.5*\myscale,0) {$H$};
               \node (G40) at (3.25*\myscale,0) {$1$};
               \node (G40) at (5.75*\myscale,0) {$1$};
               \node (G40) at (9.75*\myscale,0) {$1$};
               \node (G) at (0,\myscale) {$1$};
               \node (H) at (\myscale,\myscale) {$1$};
               \node (G40) at (5.75*\myscale,\myscale) {$1$};
               \node (G40) at (9.75*\myscale,\myscale) {$1$};
            \end{tikzpicture}
            \caption{$z=1$}
            \label{fig:z=1}
         \end{subfigure}
         \begin{subfigure}{\textwidth}
            \centering 
            \begin{tikzpicture}
               \pgfmathsetmacro{\myscale}{0.65}
               \node (ptx) at (-0.328,-0.2) {};
               \node (pty) at (-0.2,-0.328) {};
               \node (x) at (0.4 + 9.75*\myscale,-0.2) {};
               \node (y) at (-0.2, 0.45 + 3*\myscale) {};
               \draw[->] (ptx) -- (x);
               \draw[->] (pty) -- (y);
            
               \node (x0) at (0,-0.5) {0};
               \node (x1) at (\myscale,-0.5) {$e_1$};
               \node (x5) at (5.75*\myscale,-0.5) {$2e_1$};
               \node (x5) at (9.75*\myscale,-0.5) {$3e_1$};
               \node (y0) at (-0.5,0) {0};
               \node (y1) at (-0.5,\myscale) {1};
               \node (y2) at (-0.5,2*\myscale) {2};
               \node (y3) at (-0.5,3*\myscale) {3};
               
               \node (G) at (0,0) {$G$};
               \node (H) at (\myscale,0) {$G$};
               \node (G3) at (1.75*\myscale,0) {$J_1$};
               \node (G4) at (2.5*\myscale,0) {$J_2$};
               \node (G40) at (3.25*\myscale,0) {$J_3$};
               \node (G40) at (4*\myscale,0) {$J_4$};
               \node (G40) at (4.75*\myscale,0) {$\gamma_2$};
               \node (G40) at (5.75*\myscale,0) {$\gamma_2$};
               \node (G40) at (6.5*\myscale,0) {$H$};
               \node (G40) at (7.25*\myscale,0) {$H$};
               \node (G40) at (8*\myscale,0) {$H$};
               \node (G40) at (8.75*\myscale,0) {$1$};
               \node (G40) at (9.75*\myscale,0) {$1$};
               \node (G) at (0,\myscale) {$K_1$};
               \node (H) at (\myscale,\myscale) {$\gamma_2$};
               \node (G3) at (1.75*\myscale,\myscale) {$\gamma_2$};
               \node (G4) at (2.5*\myscale,\myscale) {$H$};
               \node (G40) at (3.25*\myscale,\myscale) {$H$};
               \node (G40) at (4*\myscale,\myscale) {$H$};
               \node (G40) at (4.75*\myscale,\myscale) {$1$};
               \node (G40) at (5.75*\myscale,\myscale) {$1$};
               \node (G40) at (9.75*\myscale,\myscale) {$1$};
               \node (G) at (0,2*\myscale) {$\gamma_2$};
               \node (H) at (\myscale,2*\myscale) {$1$};
               \node (G40) at (5.75*\myscale,2*\myscale) {$1$};
               \node (G40) at (9.75*\myscale,2*\myscale) {$1$};
               \node (14) at (0,3*\myscale) {$1$};
               \node (H) at (\myscale,3*\myscale) {$1$};
               \node (G40) at (5.75*\myscale,3*\myscale) {$1$};
               \node (G40) at (9.75*\myscale,3*\myscale) {$1$};
            \end{tikzpicture}
            \caption{$z=0$}
            \label{fig:z=0}
         \end{subfigure}
      \end{subfigure}
      \caption{We plot the $(M, G)$-filter from Example~\ref{ex:genus3}. We
      construct the lattice $\Lat(\phi)$ generated by $\im(\phi)$ in
      Figure~\ref{fig:filter-lattice}. Since $M=\N^2\times\N\times\N$ is ordered
      by the direct product ordering of three total orders, we can plot $\phi$
      on a three axes. In Figures~\ref{fig:z=2},~\ref{fig:z=1},
      and~\ref{fig:z=0}, we plot $\phi$ given $(x, y, z)\in M$ for fixed
      $z$-values.}
      \label{fig:genus3}
   \end{figure}
\end{ex}

\section*{Acknowledgements}

The author is indebted to J.~B.~Wilson for encouraging this research and
providing endless feedback along the way. We also thank P.~A.~Brooksbank,
A.~Hulpke, and T.~Penttila for many helpful discussions. We are grateful for the
anonymous reviewer whose thorough comments have improved the quality of the
paper.

\end{document}